\documentclass[]{article}
\usepackage{geometry}
\usepackage{amsmath}
\usepackage{amssymb}

\usepackage{bigfoot}

\usepackage{enumerate}

\usepackage{fancyhdr}
\usepackage[hidelinks]{hyperref} 

\usepackage{amsthm} 
\newtheorem{theorem}{Theorem}
\numberwithin{theorem}{section}
\newtheorem*{theorem*}{Theorem}

\newtheorem{proposition}[theorem]{Proposition}
\newtheorem{lemma}[theorem]{Lemma}

\newtheorem{remark}[theorem]{Remark}

\newtheorem*{example*}{Example}

\newtheorem{innercustomgeneric}{\customgenericname}
\providecommand{\customgenericname}{}
\newcommand{\newcustomtheorem}[2]{%
	\newenvironment{#1}[1]
	{%
		\renewcommand\customgenericname{#2}%
		\renewcommand\theinnercustomgeneric{##1}%
		\innercustomgeneric
	}
	{\endinnercustomgeneric}
}

\newcustomtheorem{customthm}{Theorem}
\newcustomtheorem{customlemma}{Lemma}

\usepackage{mathrsfs}

\usepackage{xcolor} 

\usepackage{caption}
\usepackage{subcaption}
\usepackage{graphicx}

\usepackage{algorithm}
\usepackage{algpseudocode} 

\usepackage{bbm}

\usepackage{bm}

\usepackage{mathtools}

\usepackage[toc,page]{appendix}

\usepackage{caption}
\usepackage{subcaption}

\usepackage{tikz}
\usetikzlibrary{fit,shapes,arrows,positioning}

\tikzset{decision/.style={diamond, draw, text width=4.5em, text badly centered, inner sep=0pt}}
\tikzset{inout/.style={ellipse, draw, text width=7em, text centered, rounded corners,
		minimum width=3.5cm}}
\tikzset{block/.style={rectangle, draw, text width=12em, text centered, rounded corners,
		minimum width=3.5cm}}
\tikzset{block1/.style={rectangle, draw,fill=gray!20, text width=10em, text centered, rounded corners,
		minimum width=3.5cm}}
\tikzset{line/.style={draw, -latex}}

\usepackage[hidelinks]{hyperref} 
\usepackage{url}

\usepackage{amssymb}
\usepackage{amsmath}
\usepackage{subcaption}
\newtheorem{tm}{Theorem}[section]
\newtheorem{rk}{Remark}[section]

\newtheorem{df}{Definition}[section]

\newtheorem{lm}{Lemma}[section]
\newtheorem{corollary}{Corollary}[section]

\usepackage{enumitem} 
\newcommand{\E}{\mathbb E}

\newcommand{\bi}{\mathbf i}
\newcommand{\bs}{\mathbf s}

\newcommand{\<}{\langle}
\renewcommand{\>}{\rangle}

\usepackage{xcolor}

\usepackage[toc,page]{appendix}

\usepackage{enumerate}

\usepackage{bm}

\usepackage{mathdots}

\newcommand\blfootnote[1]{%
	\begingroup
	\renewcommand\thefootnote{}\footnote{#1}%
	\addtocounter{footnote}{-1}%
	\endgroup
}

\begin{document}

    \title{A Wong--Zakai resonance-based integrator for the nonlinear Schr\"odinger equation with white noise dispersion}



\author{Jianbo Cui\thanks{jianbo.cui@polyu.edu.hk} \and Georg Maierhofer\thanks{gam37@cam.ac.uk}}
\date{\footnotemark[1]\ Department of Applied Mathematics, The Hong Kong Polytechnic
	University\\\ \vspace{-0.2cm}\\
	\footnotemark[2]\ Department of Applied Mathematics and Theoretical Physics, University of Cambridge\\\ \vspace{0.2cm}\\\today}
\lhead{}
\chead{Resonance-based method for the NLSE with white noise dispersion}
\rhead{}
\lfoot{}
\cfoot{\textsc{\thepage}}
\rfoot{\textsc{J. Cui \& G. Maierhofer}}



\maketitle

\begin{abstract}
We introduce a novel approach to the numerical approximation of the nonlinear Schr\"odinger equation with white noise dispersion in the regime of low-regularity solutions. Approximating such solutions in the stochastic setting is particularly challenging due to randomized frequency interactions and presents a compelling challenge for the construction of tailored schemes. In particular, we design the first resonance-based schemes for this equation, which achieve provable convergence for solutions of much lower regularity than previously required. A crucial ingredient in this construction is the Wong--Zakai approximation of stochastic dispersive system, which introduces piecewise linear phases that capture nonlinear frequency interactions and can subsequently be approximated to construct resonance-based schemes. We prove the well-posedness of the Wong--Zakai approximated equation and establish its proximity to the original full stochastic dispersive system. Based on this approximation, we demonstrate an improved strong convergence rate for our new scheme, which exploits the stochastic nature of the dispersive terms. Finally, we provide numerical experiments underlining the favourable performance of our novel method in practice.
\end{abstract}
\blfootnote{\textbf{2020 Mathematics Subject Classification.} \textit{Primary:} 65M12, 65M75, 35R60, 35Q41, 35Q55.}\blfootnote{\textbf{Key words and phrases.} Nonlinear Schr\"odinger equation, white noise dispersion, low regularity, Wong--Zakai approximation, resonance-based scheme.}

	\section{Introduction} 

The nonlinear Schr\"odinger equation (NLSE) with random dispersion,
\begin{align}\label{nls-random}
\frac {d v}{dt}(t)=\bi \frac 1 \epsilon m\!\left(\frac t{\epsilon^2}\right)\Delta v(t)+ \bi \lambda |v(t)|^2v(t),\quad v(0)=u_0,   
\end{align}
with  $t>0$, $\lambda\neq 0$ and $\epsilon>0$,
is a model describing the propagation of a signal in an optical fibre with dispersion management (see, e.g., \cite{Agr21}).
Here $\bi$ denotes the imaginary unit, and $m$ is a real-valued centered stationary continuous random process, $\Delta$ is the Laplacian operator on the 1D torus $\mathbb T=\mathbb{R}/(2\pi\mathbb{Z})$, and  $u_0:x \in \mathbb T \mapsto u_0(x)\in \mathbb C$ is a complex-valued  function. 
In fiber optics, $x$ is used  to denote the retarded time, $t$ is used to denote the distance along the fibre,
and the nonlinearity models the nonlinear response of the random medium to the electric field. 
Under certain ergodic assumptions on the random process $m$, it has been  shown in \cite{MR2652190} that the limiting equation of \eqref{nls-random} when $\epsilon \to 0$ is the stochastic NLSE with white noise dispersion,
\begin{align}\label{model2}
du(t)=\bi \Delta u(t)\circ dB(t)+\bi \lambda |u(t)|^2 u(t)dt,\quad u(0)=u_0,
\end{align}
where $B$ is the standard Brownian motion defined on a completed filtered probability space $(\Omega,\mathcal F, \{\mathcal F_t\}_{t\ge 0},\mathbb P)$, and $\circ$ denotes the Stratonovich product.

Since in general \eqref{nls-random} and \eqref{model2} do not have closed-form solutions, numerical methods for this system have received an increasing amount of attention in recent years: Eq. \eqref{nls-random} has been initially studied in \cite{MR4278943} by a split step numerical scheme, where the cubic nonlinearity is replaced by a nicer Lipschitz function such that a good approximation error can be established. In \cite{MR4167043}, randomized exponential integrators have been proposed for modulated nonlinear Schr\"{o}dinger equations, i.e., \eqref{nls-random} with  $\frac 1{\epsilon}m(\frac t{\epsilon^2})$ replaced by a general time-dependent function $g$. In later work, \cite{belaouar2015numerical} proposed a semi-discrete Crank--Nicolson scheme and derived its convergence order in probability for the white noise dispersion case (i.e., \eqref{model2}). Inspired by  the corresponding numerical tests, \cite{MR2832639} further obtained the well-posedness of the 1D quintic NLSE with white noise dispersion. We refer to \cite{Ste24,MR4288117,MR4250287,MR4755536} and references therein for more recent theoretical  results.
The stochastic symplectic and multi-symplectic structure of \eqref{model2} were discovered in \cite{cui2017stochastic}, where the symplectic and multi-symplectic integrators were also developed for \eqref{model2}. The authors in \cite{cohen2017exponential}
considered an explicit  exponential integrator for \eqref{model2} and derived  its convergence order in probability. Later, a splitting scheme for Schr\"odinger equation with nonlocal interaction cubic nonlinearity and white noise dispersion  was analyzed in \cite{MR4400428}.
A  methodology based on the multi-revolution idea was used in \cite{MR4048623} to propose a weakly convergent scheme for solving a highly oscillatory NLSE with white noise dispersion. 

In all prior work, it was observed that deriving convergence of order one in the Sobolev  $H^{\bs}$-norm with $\bs>\frac 12$ of numerical schemes for \eqref{model2}  usually requires initial data $u_0\in H^{\bs+4}.$ This is problematic because it prevents the accurate simulation of low-regularity solutions which is in contrast to theoretical analysis showing that \eqref{model2} is globally well-posed for initial data in $L^2$ (cf. \cite{MR2652190} and \cite{Ste24}). In particular, if $u_0\in H^{\bs+4}$, for a given numerical scheme of \eqref{model2} with numerical solution $\{u^n\}_{n=0}^{K}$ with $K\in \mathbb N^+$, one may expect linear convergence order in probability with the aforementioned methods, i.e.,
\begin{align*}
  \lim_{\mathcal M\to\infty} \mathbb P(\|u(t_n)-u^n\|_{H^{\bs}}\ge  \mathcal M\tau)=0, \quad 0\le n\le K,  
\end{align*}
or linear convergence order in the pathwise sense, i.e.
\begin{align*}
   \|u(t_n)-u^n\|_{H^{\bs}}\le \mathfrak{R} \tau,\quad 0\le n\le K.
\end{align*}
 Here $\tau=\frac {T}K$ is the the timestep size with $T$ being the terminal time, $t_n=n\tau$ and $\mathfrak R$ is a random variable that depends on $H^{s+4}$-norm of the solution $u(\cdot)$. Formally speaking, to let the solution $u(\cdot)$ of \eqref{model2}
enjoy the $\frac 12$-H\"older regularity in time under the $H^{\bs}$-norm, one needs the $H^{\bs+2}$-regularity of $u(\cdot)$ in space such that the right-hand side is finite. As a consequence, a traditional numerical scheme for \eqref{model2} is expected to be of order $\frac 12$ if $u_0\in H^{\bs+2}$ 
since the increment of Brownian motion is roughly of the size $\sqrt{\tau}$. To derive first order convergence, one needs more regularity assumption such as $u_0\in H^{\bs+4}.$ 

This phenomenon is similar to that in the deterministic case where one usually needs the initial data $u_0\in H^{\bs+2}$ such that the  numerical scheme is convergent of order $1$ under the $H^{\bs}$-norm. 
To reduce the regularity assumption on the initial data, \cite{ostermann2018low} introduced a low regularity exponential-type integrator
which is shown to possess first-order convergence in $H^{\bs}$ under the low regularity assumption $u_0\in H^{\bs+1}.$ For deterministic NLSEs, this low regularity assumption can be even further relaxed in some special situations \cite{ostermann2018low,MR4312402}. Such low-regularity integrators have turned out to be a highly versatile class of numerical methods that can be applied to a wide range of dispersive systems \cite{alama2023symmetric,bronsard2023error,banicamaierhoferschratz22,feng_maierhofer_schratz_2023,alamabronsard_et_al23,bruned_schratz_2022,feng2024explicit,rousset2021general}. Thus, it is natural to study the influence of random noise on low-regularity integrators for the stochastic NLSE, such as \eqref{model2}.

Recently, there are some numerical works focusing on low-regularity integrators for a stochastic NLSE driving by a random force, 
\begin{align}\label{model-snls}
 du=\bi\Delta udt+\bi \lambda|u|^2u dt+\bi \sigma(u)dW(t),   
\end{align}
where $\sigma(u)=1$ (corresponds to the additive noise case) or $\sigma(u)=u$ (corresponds to the linear multiplicative noise case), and $W$ is a $Q$-Wiener process. For instance, in \cite{arXiv:2312.16690}, the authors construct low-regularity integrators for \eqref{model-snls} which possess pathwise convergence order $\frac 12$ under the $L^2$-norm if $u(0)\in H^1$ and $W(t)\in H^2$. Under additional regularity assumptions that $u_0\in H^2$ and $W(t)\in H^4$, a first order scheme is also developed for \eqref{model-snls} driven by a linear multiplicative noise. 
For \eqref{model-snls} driven by the additive noise, the authors in \cite{arXiv:2410.22201} construct a perturbed first-order non-resonant low-regularity integrator under $H^{\bs}$-norm, $\bs>\frac 12$, if $u_0\in H^{\bs+1/2}$ and $W(t)\in H^{\bs+2}$, and analyze its 
long-term error under a smallness condition on  $u_0$ and $W$. 

However, to the best of our knowledge, there has been no prior work on the construction of low-regularity integrators for \eqref{model2}.  
Compared with \eqref{model-snls}, the numerical analysis of low-regularity integrators for \eqref{model2} has intrinsic differences and provides new, interesting challenges. Firstly, for \eqref{model-snls}, with the help of the moment bounds and exponential integrability (see, e.g., \cite{MR3826675}), one can study the strong convergence for  numerical schemes - allowing for stronger statements than convergence in probability or pathwise convergence. In contrast, since there is no a uniform moment estimate for \eqref{model2}  under the Sobolev norm $H^{\bs},\bs>0$ (cf.  \cite{MR2652190,Ste24}), existing numerical results are  mainly focused on the convergence in probability or pathwise convergence. 
Secondly, different from \eqref{model-snls},
 due to the effect of random dispersion, the minimum regularity condition on $u_0$ required to derive first order convergence of the low-regularity integrator is still unclear.  
Finally, in the construction of low regularity integrators,  one can use resonance-based techniques as in the deterministic case \cite{ostermann2018low} since the resonance-based part for the cubic nonlinearity in \eqref{model-snls} is the same.   For \eqref{model2}, we can not directly use the techniques as in the deterministic case 
since the many highly oscillatory random integrals emerging from the two-parameter group $e^{\bi \Delta(B(t)-B(s))}, t,s\ge 0,$ cannot be solved exactly.

To address these issues, we develop the following techniques for the construction of a new low-regularity integrator for \eqref{model2}, the stochastic dispersion low-regularity integrator (SDLRI) \eqref{eqn:1st_order_lri}, which are the central contributions of this present work:

\begin{itemize}
    \item  Wong--Zakai approximation: we introduce a regularized approximation of the original model \eqref{model2} via a linear Wong--Zakai approximation of the Brownian motion  \cite{MR183023}. 
    Furthermore, we 
  prove the well-posedness and mass conservation law of this regularized model.   
    \item Computing the oscillatory random integral: When devising the low-regularity integrator of \eqref{model2}, it is essential to consider the term
    \begin{align*}
\int_0^{\tau}e^{\bi(B(t_n+s)-B(t_n))(k^2+k_1^2-k_2^2-k_3^2)}ds
    \end{align*}
for all $n\le K$ with $\tau=\frac T K$ and all $k+k_1=k_2+k_3, (k,k_1,k_2,k_3)\in \mathbb Z^4.$ Since it is hard to directly calculate this term and the corresponding dominant interaction term $\int_0^{\tau}e^{\bi(B(t_n+s)-B(t_n))2k_1^2}ds$, we make use of the linear Wong--Zakai approximation to compute these resonance terms. As a result, the proposed numerical scheme \eqref{model2} can be viewed as a low-regularity integrator for the regularized model (see section \ref{sec-3}). This strategy also provides a basis for the systematic construction of resonance-based
schemes for stochastic partial differential equations with white noise dispersion in future work. 
    
\item Strong convergence error analysis: 
We focus on the strong convergence error of numerical methods  for \eqref{model2}, which is 
different from \cite{MR4400428,cohen2017exponential,cui2017stochastic,belaouar2015numerical} and references therein. More precisely, under a smallness condition on the seminorm of $u_0$, we show that (see Corollary \ref{cor-1} for details)
\begin{align}\label{main-err}
    &\sup_{n\le K}\left\|u(t_n)-u^{\delta,R,N}_n\right\|_{L^2(\Omega;H^{\bs})} \lesssim  \tau^{\max(\frac {\gamma+1}4,1)}+\delta^{\max (\frac{\gamma}4,1)}+N^{-\gamma}, 
\end{align}
where $\tau=\frac {T}{K}>0$ is the time stepsize, $\delta>0$ is the width of Wong--Zakai approximation, $R>0$ is a truncation parameter and $N^+=\in\mathbb N$ is the parameter of the Galerkin projection.
Here 
$\{u_n^{\delta,R,N}\}_{n\le K}$  denotes the numerical solution of the proposed scheme and $u_0\in H^{\bs+\gamma}$ with $\gamma>0.$
Compared with pathwise convergence or convergence in probability, the above error estimate is more useful in the Monte Carlo method to quantify the simulation error.
To the best of our knowledge,  this is the first result on the strong convergence order of numerical methods  for \eqref{model2}. 

\item  Interesting order reduction phenomenon of low-regularity integrator that is  intrinsically different from the deterministic case. To illustrate this clearly, we take $\delta$ small enough such that the error of Wong--Zakai approximation can be ignored. 
Then from \eqref{main-err}, to derive the first order strong convergence w.r.t $\tau$ for the proposed scheme, we need $H^{\bs+3}$-regularity condition on $u_0$ (i.e., $\gamma=3$). 
In contrast, if one performs the pathwise convergence analysis as in subsection \eqref{path-subsec}, it can be seen that the pathwise converge order w.r.t. $\tau$ is $\frac 12$ if $u_0\in H^{\bs+\gamma}$ for any $\gamma\in \mathbb N^+.$

\end{itemize}

The remainder of this article is organized as follows. In section \ref{sec-2}, we introduce the Wong--Zakai approximation of \eqref{model2}, and present its well-posedness and approximation error. In section \ref{sec-3}, we develop the low-regularity integrator of the Wong--Zakai approximation and show its error analysis. Finally, some numerical tests are shown in section \ref{sec-4} to verify our findings and concluding remarks are provided in section~\ref{sec:conclusions}.

\section{Wong--Zakai approximation}
\label{sec-2}
In this section, we present a linear Wong--Zakai approximation  of \eqref{model2}, which will provide an efficient way to compute the oscillatory integral when designing the low-regularity resonance-based schemes in section \ref{sec-3}. For other ways of construction Wong--Zakai approximations, we refer the reader to \cite{SL2017JDE} and references therein. 

For this we begin by defining the truncated Brownian motion $B^R$ satisfying that for any $t\ge 0,$ $B^{R}(t):=B(t)$ if $|B(t)|\le R\sqrt{t}$, $B^{R}(t)=R\sqrt{t}$ if $B(t)\ge R\sqrt{t}$, and $B^{R}(t)=-R\sqrt{t}$ if $B(t)\le -R\sqrt{t}.$
Then we define a linear interpolant $B^{\delta,R}$ as  
\begin{align}\label{eqn:linear_WZ_approx}
B^{\delta,R}(t):=B^{R}(\lfloor t\rfloor_{\delta})+\frac{t-\lfloor t\rfloor_{\delta}}{\delta}(B^R(\lfloor t\rfloor_{\delta}+\delta)-B^R(\lfloor t\rfloor_{\delta}))
\end{align}
where  the width $\delta>0$ and the truncation parameter $R\in [1,+\infty]$. 
Here $\lfloor t\rfloor_{\delta}:=\delta\lfloor t/\delta\rfloor$ with 
$\lfloor(\cdot)\rfloor$ denoting the integer part of $(\cdot)$. When $R=+\infty,$ $B^{\delta,R}$ is the standard linear Wong--Zakai approximation of the  Brownian motion $B$  \cite{MR183023}.
Introducing this truncated number $R$ will be useful in later sections for obtaining a priori $L^2$-estimates of our new numerical scheme. We also use $a \lesssim b$ to denote $a \le C b$ for $a,b\in \mathbb R$ with some generic constant $C>0$ independent of $R,\delta,$ the sample $\omega\in \Omega$, time stepsize $\tau$ and spatial parameter $N\in \mathbb N^+.$

\subsection{Wong--Zakai approximation of truncated Brownian motion}\label{sub-sec-wong-zakai}
We start with the approximation error of the Wong--Zakai approximation \eqref{eqn:linear_WZ_approx}.

\begin{lm}\label{lm-2.1-wong-zakai}
Let $T>0$, $p\ge 1$ and $\delta\in (0,1)$. There exists $R_0=\max\big(\sqrt{4p|\ln(\delta)|}, $ $ pe^{-1}\big)>0$ such that for any $R\ge R_0$ and any $t\in [0,T],$ it holds that \begin{align}\label{err-wk}
     \|B^{\delta,R}(t)-B(t)\|_{L^p(\Omega)}\lesssim \sqrt{\delta}.   
    \end{align}
\end{lm}
 \begin{proof}
By the triangle inequality, 
we  have 
\begin{align*}
\|B^{\delta,R}(t)-B(t)\|_{L^p(\Omega)}
&\le \|B^{R}(t)-B(t)\|_{L^p(\Omega)}+\|B^{\delta,R}(t)-B^R (t)\|_{L^p(\Omega)}.
\end{align*}From the definition of $B^R$, taking $R\ge pe^{-1}$, we have 
 \begin{align}\label{err-trun}
& \|B^{R}(t)-B(t)\|_{L^p(\Omega)}\le \Big(\int_{|x|\ge R\sqrt{t}} \frac 1{\sqrt{2\pi t}}e^{-\frac {|x|^2}{2t}} |x-R\sqrt{t}|^p dx\Big)^{\frac 1p}\\\nonumber
 &\le \Big(\int_{|y|\ge R} \frac 1{\sqrt{2\pi }}e^{-\frac {|y|^2}{2}} |y-R|^p dy\Big)^{\frac 1p}\\\nonumber
 &\le \Big(e^{-\frac{R^2}2}\int_{0}^{+\infty} \frac 2{\sqrt{2\pi }}e^{-\frac {|y|^2}{2}} e^{-Ry}|y|^p dy\Big)^{\frac 1p}< e^{-\frac {R^2}{2p}}\lesssim \delta^2.
 \end{align}
 
On the other hand, using again the triangle inequality and \eqref{err-trun},  we find that there exists some $C_1>0$ such that \begin{align*}
 &\|B^{\delta,R}(t)-B^R(t)\|_{L^p(\Omega)}\\
 \nonumber &\le \|B^R(t)-B^R(\lfloor t\rfloor_{\delta})\|_{L^p(\Omega)}+\|B^R(\lfloor t\rfloor_{\delta}+\delta)-B^R(\lfloor t\rfloor_{\delta})\|_{L^p(\Omega)}\\\nonumber
 &\le \|B(t)-B(\lfloor t\rfloor_{\delta})\|_{L^p(\Omega)}+\|B(\lfloor t\rfloor_{\delta}+\delta)-B(\lfloor t\rfloor_{\delta})\|_{L^p(\Omega)}\\\nonumber
&+8R\sqrt{T}\Big(\int_{|x|\ge R}\frac 1{\sqrt{2\pi}}e^{-\frac {|x|^2 }2}dx\Big)^{\frac 1p}\\\nonumber 
&\le C_1 \sqrt{\delta}+C_1 R\sqrt{T}e^{-\frac {R^2}{4p}}.
\end{align*}
where in the last inequality we also use the standard $1/2$-H\"older continuity property of Brownian motion in $L^p(\Omega)$-norm, i.e., $\|B(t)-B(s)\|_{L^p(\Omega)}\lesssim |t-s|^{\frac 12}$.
Letting $R\ge R_0=\max\big(\sqrt{4p|\ln(\delta)|}, pe^{-1}\big)$, we  obtain 
\begin{align}\label{err-trun-est}
  \|B^{\delta,R}(t)-B^R(t)\|_{L^p(\Omega)} \lesssim \sqrt{\delta}.
\end{align}
 We  complete the proof by combining \eqref{err-trun} and \eqref{err-trun-est} together.
 \end{proof}

As a direct consequence of Lemma \ref{lm-2.1-wong-zakai}, we have 
  \begin{align}\label{hold-est}
   \|B^{\delta,R}(t)-B^{\delta,R}(s)\|_{L^p(\Omega)}\lesssim \max(|t-s|^{\frac 12},\delta^{\frac 12}). 
 \end{align}
 For completeness, a proof of \eqref{hold-est} is provided in Appendix~\ref{app:proof_of_hold-est}.



\subsection{Wong--Zakai approximation of the NLSE with white noise dispersion}

Here we prove the well-posedness of its Wong--Zakai approximation, i.e., 
\begin{align}\label{wong-zakai}
du^{\delta,R,N}(t)=\bi \Delta u^{\delta,R,N}(t)\circ dB^{\delta,R}(t)+\bi \lambda \pi_N (|u^{\delta,R,N}(t)|^2 u^{\delta,R,N}(t))dt
\end{align}
under a smallness condition for the initial condition $u^{\delta,R,N}(0)=\pi_N u_0$. Here $N\in \mathbb N^+\cup\{+\infty\}$ and $\pi_N$ denotes the Galerkin projection up to the frequency $|k|\le N$.
We use the convention that $\pi_{+\infty}=I$, $u^{0,+\infty,+\infty}(t)=u(t).$ 

We would like to mention that following the approach in \cite{Ste24}, one can obtain the well-posedness result without the smallness condition. But this approach without the smallness assumption fails to yield the uniform a-priori regularity estimate of $u^{\delta,R,N}$ w.r.t. $\delta,R,N.$ For completeness, we include the proof  without the smallness condition in Appendix~\ref{sec-appendix-2}. The definition of mild solutions is also presented in Appendix~\ref{sec-appendix-1}.

Let $\bs\in \mathbb N.$
Denote the standard Sobolev space by $H^{\bs}:=H^{\bs}(\mathbb T)$ (denote $L^q:=L^q(\mathbb T)$ with $q\in [1,+\infty]$ for simplicity) 
equipped with the norm 
\begin{align*}
\|\cdot\|_{H^{\bs}}&:=\sqrt{\|\cdot\|_{L^2}^2+\|\cdot\|_{\dot H^{\bs}}^2}:=\sqrt{\sum_{m\in \mathbb Z}|\<\cdot,e_m\>|^2+|\lambda_m|^{2\bs} |\<\cdot,e_m\>|^2}.
\end{align*}Here $(\lambda_m,e_m)_{m\in \mathbb Z}$ is the eigensystem of the Laplacian operator (indeed, $\lambda_m=m^2$), and $\dot H^{\bs}$ 
denotes the homogeneous Sobolev space    consisting of the function $f=\sum\limits_{m\neq 0}f_me_m$ equipped with the norm $\|\cdot\|_{\dot H^{\bs}}.$ 


\begin{theorem}
\label{small-wel}
Let $T>0$, $\bs\in \mathbb N^+$ and $u_0\in L^2.$ There exists a constant  $\epsilon>0$ such that if $\|u_0\|_{\dot H^{\bs}}\le \epsilon,$
then  for  \eqref{wong-zakai} with any $\delta\in [0,1)$, $R\in [1,+\infty]$ and $N\in \mathbb N^+\cup\{+\infty\}$, there exists a unique mild solution in $\mathcal C([0,T],H^{\bs})$ such that 
\begin{align}\label{pri-hs}
\|u^{\delta, R,N}\|_{\mathcal C([0,T];H^{\bs})}\leq 2\|u_0\|_{H^{\bs}}+2\epsilon, \;\text{a.s.},
\end{align}
and that, for any $t\in [0,T],$ 
\begin{align}\label{pri-mass-con}
 \|u^{\delta, R,N}(t)\|_{L^2}=   \|\pi_N u_0\|_{L^2}, \;\text{a.s.} 
\end{align}

\end{theorem}

\begin{proof}
The proof consists of the following two steps. 

\textbf{Step 1}: 
Denote $\mathcal X=\pi_N L^2\cap H^{\bs}$ equipped with the $H^{\bs}$-norm. Here $\pi_N L^2$ is the linear space generated by $\{e_i\}_{ |i|\le N}$.
We  define a map $\Gamma$  by
\begin{align*}
\Gamma(\psi)(t):=e^{\bi \Delta B^{\delta,R}(t)}\pi_N u_0+\bi \int_0^t \lambda e^{\bi \Delta (B^{\delta,R}(t)-B^{\delta,R}(s))} \pi_N |\psi(s)|^2\psi(s)ds
\end{align*}
for any process $\psi \in \mathcal C([0,T_1],\mathcal X)$ with some $0<T_1\le T$ being determined later.  We claim that $\Gamma$ is well-defined. Fix $\mathcal Q>1$ to be determined later and consider the ball $$\mathcal B_{\mathcal Q}:=\{\psi\in \mathcal C([0,T_1],\mathcal X)  \;|\; \|\psi\|_{\mathcal C([0,T_1],\mathcal X)}\le \mathcal Q\}.$$
Note the unitary property of $e^{\bi \Delta (B^{\delta,R}(t)-B^{\delta,R}(s))}$ holds for any $s,t\ge 0$. 
By further using the algebra property of $\dot H^{\bs}$ (since $d=1, \bs\in \mathbb N^+$), and the Gagliardo--Nirenberg interpolation inequality,
\begin{align}\label{GN-ine}
\|f\|_{L^{\infty}}\le C_{GN} \|f\|_{\dot H^{\bs}}^{\frac 1{2\bs}}\|f\|_{L^2}^{\frac {2\bs-1}{2\bs}},\quad \text{for}    \; f\in H^{\bs},
\end{align}
with a constant $C_{GN}>0$,
we obtain that there exists a constant $C_{Sob,1}>0$ such that for any $\psi\in \mathcal B_{\mathcal Q},$
\begin{align}\nonumber
\|\Gamma(\psi)\|_{\mathcal C([0,T_1],H^\bs )}&\le \|u_0\|_{H^\bs}+C_{Sob,1} |\lambda|\int_0^{T_1} \|\psi\|^2_{L^{\infty}}\|\psi\|_{ H^\bs}ds\\\label{wel-map}
&\le \|u_0\|_{ H^\bs}+C_{Sob,1}C_{GN}^2 |\lambda|T_1  \|\psi\|_{\mathcal C([0,T_1],\dot H^{\bs})}^{1+\frac 1\bs} \|\psi\|_{\mathcal C([0,T_1],L^2)}^{2-\frac 1{\bs}}.
\end{align}
As a consequence, if $\|u_0\|_{H^\bs}\le \frac {\mathcal Q }2,$ and $C_{Sob,1}C_{GN}^2|\lambda|T_1 {\mathcal Q}^2\le \frac 12$, then $\Gamma$ maps $\mathcal B_\mathcal Q$ into itself. 

Next we prove the existence of the unique fixed point for $\Gamma.$
For any $\psi_1,\psi_2\in \mathcal B_\mathcal Q,$ similarly to \eqref{wel-map}, we get for some positive constant $C_{Sob,2}>0,$  
\begin{align}\label{wel-map-lip}
&\|\Gamma(\psi_1)-\Gamma(\psi_2)\|_{\mathcal C([0,T_1],  H^\bs)}\\\nonumber
&\le C_{Sob,2}|\lambda|T_1 (\|\psi_1\|_{\mathcal C([0,T_1],L^{\infty})}^2+\|\psi_2\|_{\mathcal C([0,T_1],L^{\infty})}^2) \|\psi_1-\psi_2\|_{\mathcal C([0,T_1],  H^\bs )}\\\nonumber
&\le 2|\lambda|C_{Sob,2}C_{GN}^2T_1 \mathcal Q ^{2}\|\psi_1-\psi_2\|_{\mathcal C([0,T_1], H^\bs)}.
\end{align}
By letting $\|u_0\|_{H^\bs }\le \frac {\mathcal Q}2,$  $C_{Sob,i}C_{GN}^2|\lambda|T_1 \mathcal Q^2< \frac 12$ with $i=1,2$,  we obtain the unique fixed point denoted by $u^{\delta, R,N}\in \mathcal C([0,T_1];\mathcal X)$. It is also  a local  solution  of \eqref{wong-zakai} 
before $T_1= \inf\limits_{i=1,2} \frac 1{|\lambda|\mathcal Q^2C_{GN}^2C_{Sob,i}^2}$ with  $\mathcal Q\ge 2\|u_0\|_{H^{\bs}}.$

\textbf{Step 2:} Now we extend the local solution to a global solution on $[0,T]$ under the condition that $\|u_0\|_{\dot H^{\bs}}\le \epsilon $ with $\epsilon$ being determined later. We define the time $T_2=\inf\{t\in (0,T]|  \|u^{\delta,R,N}(t)\|_{\dot H^{\bs}}>2 \epsilon \}$. 
Similarly to \eqref{wel-map}, we have the following upper bound in the homogeneous Sobolev norm, for any $t\le T_2,$
\begin{align*}
&\|u^{\delta, R, N}\|_{\mathcal C([0,t],\dot H^\bs )}\le \|u_0\|_{\dot H^\bs}+C_{Sob,1}C_{GN}^2 |\lambda|t  \|u^{\delta,R, N}\|_{\mathcal C([0,t],\dot H^{\bs})}^{1+\frac 1\bs} \|u^{\delta,R, N}\|_{\mathcal C([0,t],L^2)}^{2-\frac 1{\bs}}.
\end{align*}
Note that the mass conservation law \eqref{pri-mass-con}  holds for $t\in [0,T_2]$ thanks to the chain rule. 
As a consequence, for $\epsilon$ satisfying $$2^{1+\frac 1\bs }C_{Sob,1}C_{GN}^2 |\lambda|T\|u_0\|_{L^2}^{2-\frac 1{\bs}}\epsilon^{\frac 1{\bs}}<1$$ it holds that for any $t\in [0,T_2],$ 
\begin{align*}
\|u^{\delta, R, N}\|_{\mathcal C([0,t],\dot H^\bs )}
&\le \epsilon+2^{1+\frac 1\bs }C_{Sob,1}C_{GN}^2 |\lambda|T  \|u_0\|_{L^2}^{2-\frac 1{\bs}}\epsilon^{1+\frac 1\bs}<2\epsilon.
\end{align*}
The standard bootstrap argument (see, e.g., \cite{MR2233925}) implies that $\|u^{\delta,R,N}(t)\|_{\dot H^{\bs}}\le 2\epsilon$ for any  $t\in [0,T]$. Thus, the existence of the global solution follows and  \eqref{pri-hs}-\eqref{pri-mass-con} holds. 
The uniqueness of the solution follows from the local Lipschitz property of the nonlinearity in $H^{\bs}$.
\end{proof}

\begin{rk}\label{rk-1}
We would like to mention that the smallness condition on the initial data has been often used to deal with  global well-posedness and the longtime stablity of the NLSE (cf. \cite{MR1616917}).

From the above analysis, one can see that for a given small value $\epsilon>0,$ the lifespan time $T$ of $u^{\delta,R,N}$ is at least of the size $\mathcal O\big(\|u_0\|_{L^2}^{-2+\frac 1{\bs}}\epsilon^{-\frac 1{\bs}}|\lambda|^{-1}\big)$
if $\|u_0\|_{\dot H^{\bs}}\le \epsilon.$
If  in addition assume that the standard Sobolev norm $\|u_0\|_{ H^{\bs}}\le \epsilon,$ then the lifespan time of $u^{\delta,R,N}$ is at least of the size $O(\epsilon^{-2}|\lambda|^{-1})$.

\end{rk}

\subsection{Approximation error of the Wong--Zakai approximation}
In this part, we present the  error estimate of \eqref{wong-zakai} compared to the original equation \eqref{model2}.
This ensures that the numerical scheme based on the Wong-Zakai approximation developed in the next section is consistent in a suitable sense. Below is a useful lemma whose proof is presented in Appendix~\ref{app:proof_of_lem-2.1}.

\begin{lemma}\label{lm-2.1}
Let $\bs\in \mathbb N$, $\gamma\in [0,2]$. For all $t_1,t_2\ge 0$ and $\psi\in H^{\bs+\gamma},$
    \begin{align*}
        &\| [e^{\bi \Delta (B(t_1)-B(t_2))}-e^{\bi \Delta (B^{\delta,R}(t_1)-B^{\delta,R}(t_2))}]\psi\|_{H^\bs }\\
        &\lesssim \big(|B(t_1)-B^{\delta,R}(t_1)|^{\frac{\gamma}{2}} +|B(t_2)-B^{\delta,R}(t_2)|^{\frac{\gamma}{2}}\big) \|\psi\|_{H^{\bs+\gamma}}. 
    \end{align*}
\end{lemma}

Under the smallness condition on the homogeneous Sobolev norm, we have the following pathwise error estimate.

\begin{proposition}\label{prop:approx_prop_wong_zakai_periodic} 
Let $T>0$, $\bs \in \mathbb N, \gamma\in [0,2],$  $\bs+\gamma>\frac 12$ and $u_0\in L^2$.
 There exists a small $\epsilon>0$ such that if $\|u_0\|_{\dot H^{\bs+\gamma}}\le \epsilon,$
then  for  \eqref{wong-zakai} with any $\delta\in [0,1)$, $R\in [1,\infty]$ and $N\in \mathbb N^+\cup\{+\infty\}$,
it holds that for all $t\in [0,T],$
\begin{align}\label{path-err-wk}
 &\|u-u^{\delta,R,N}\|_{\mathcal{C}([0,t];H^{\bs})}\\\nonumber
 &\lesssim  \int_0^t (|B(s)-B^{\delta,R}(s)|^{\frac{\gamma}{2}}+|B(t)-B^{\delta,R}(t)|^{\frac{\gamma}{2}})ds + N^{-\gamma}, \;\text{a.s.}
\end{align}

\end{proposition}

\begin{proof} 
    By the mild forms  of  $u$ and $u^{\delta,R,N}$, we have 
    \begin{align*}
        &u(t)-u^{\delta,R,N}(t)=e^{\bi B(t)\Delta}u_0-e^{\bi B^{\delta,R}(t)\Delta}\pi_N u_0\\
&+i\lambda\int_{0}^t e^{\bi (B(t)-B(s))\Delta} |u(s)|^{2} u(s) -e^{\bi  (B^{\delta,R}(t)-B^{\delta,R}(s))\Delta}\pi_N |u^{\delta,R,N}(s)|^{2}u^{\delta,R,N}(s) ds\\
&:=I_1+I_2, \;\text{a.s.}
    \end{align*}
Using Lemma \ref{lm-2.1} and the property of $\pi_N,$
\begin{align}\label{err-galerkin}
    \|(I-\pi_N)\psi\|_{H^{\bs}}\lesssim N^{-\gamma }\|\psi\|_{H^{\bs+\gamma}}
\end{align}
 for any $\gamma\ge0$ and $\psi\in H^{\bs+\gamma}$, it follows that 
{\small\begin{align*}
&I_1\le \|e^{\bi B(t)\Delta}u_0-e^{\bi B^{\delta,R}(t)\Delta} u_0\|_{H^{\bs}}+\|e^{\bi  B^{\delta,R}(t)\Delta}(I-\pi_N) u_0\|_{H^{\bs}}\\
&\lesssim |B(t)-B^{\delta,R}(t)|^{\frac \gamma  2}\|u_0\|_{H^{\bs+\gamma}}+N^{-\gamma}\|u_0\|_{H^{\bs+\gamma}}.
\end{align*}}For the nonlinear term, applying the algebra property of $H^{r}$ for any $r>\frac 12$, and letting $\bs+\gamma>\frac 12$, it holds that  
    \begin{align*}
    I_2
        &\leq \Big\|\int_{0}^t [e^{\bi (B(t)-B(s))\Delta} -e^{\bi (B^{\delta,R}(t)-B^{\delta,R}(s))\Delta}] |u(s)|^{2} u(s) ds\Big\|_{H^{\bs}} \\
&+\Big\|\int_{0}^t e^{\bi (B^{\delta,R}(t)-B^{\delta,R}(s))\Delta} (|u(s)|^{2} u(s) -|u^{\delta,R,N}(s)|^{2}u^{\delta,R,N}(s) )ds\Big\|_{H^{\bs}}\\
&+\Big\|\int_{0}^t e^{\bi (B^{\delta,R}(t)-B^{\delta,R}(s))\Delta}(I-\pi_N) |u^{\delta,R,N}(s)|^{2}u^{\delta,R,N}(s) ds\Big\|_{H^{\bs}}\\
        &\leq C [\int_0^t (|B(s)-B^{\delta,R}(s)|^{\frac{\gamma}{2}}+|B(t)-B^{\delta,R}(t)|^{\frac{\gamma}{2}}) ds+t N^{-\gamma}]\\
        &\quad \times \sup_{s\in[0,t]}(1+\|u(s)\|_{H^{r+\gamma}}^{3}+\|u^{\delta, R,N}(s)\|_{H^{\bs+\gamma}}^{3})\\
        &+C t\sup_{s\in[0,t]}(\|u(s)\|_{H^{\bs+\gamma}}^{2}+\|u^{\delta, R,N}(s)\|_{H^{\bs+\gamma}}^{2})\sup_{s\in[0,t]}\|u(s)-u^{\delta,R,N}(s)\|_{H^{\bs}}.
    \end{align*}
    Combining the above estimates, using   Theorem \ref{small-wel} and  Gronwall's inequality,  we obtain the desired result.
\end{proof}

By taking the second moment of \eqref{path-err-wk} and exploiting \eqref{hold-est}, it seems that the mean-square convergence rate of the Wong--Zakai approximation w.r.t $\delta>0$ under the $H^{\bs}$-norm is  $\mathcal O(\sqrt{\delta})$ for $H^{\bs+2}$ data. However, the following  shows that the mean-square convergence rate w.r.t. the width $\delta>0$ can be improved by assuming $H^{\bs+4}$-regularity on the solution. 

\begin{lemma}\label{lm-improv-order}
Let $T>0$, $\bs \in \mathbb N$, $\bs+\gamma>\frac 12$, $\gamma>0$  and $u_0\in L^2$.
 There exists a constant  $\epsilon>0$  such that if $\|u_0\|_{\dot H^{\bs+\gamma}}\le \epsilon,$
then  for  \eqref{wong-zakai} with any $\delta\in [0,1)$, $R\in [R_0,\infty]$ with  $R_0=\max\big(\sqrt{8|\ln(\delta)|}, 2e^{-1}\big),$ and $N\in \mathbb N^+\cup\{+\infty\}$, 
\begin{align*}
 \sup_{t\in [0,T]}\|u(t)-u^{\delta, R,N}\|_{L^2(\Omega; H^{\bs})}\lesssim (\delta^{\max(1,\frac \gamma 4)}+N^{-\gamma}).
\end{align*}
\end{lemma}

Since the proof of this statement relies on similar arguments to the ones used in improving the mean-square convergence order of the numerical scheme in section \ref{sec-3}, we present its proof in Appendix~\ref{app:proof_lm-improv-order}.

\section{Low-regularity resonance-based schemes}
\label{sec-3}
Based on the above analysis we now aim to construct a resonance-based numerical scheme directly for the Wong--Zakai approximation \eqref{wong-zakai} rather than for the original NLSE. We follow roughly the analysis in \cite{ostermann2018low,maierhofer2023bridging}. We will consider the behaviour of the scheme in both pathwise and mean-square sense. Let $\tau>0$ be the time stepsize. Moreover for notational simplicity we assume that $t_n/\delta\in\mathbb{Z}$, $T/\tau=K\in \mathbb N^+$ for all timesteps $t_n=n\tau, n\in\mathbb{Z}, n\le K,$ throughout this paper. We note that the construction and analysis of the resonance-based integrator can be performed analogously even when this is not the case.

Let us consider the twisted variable $v^{\delta,R,N} (t)=e^{-\bi B^{\delta,R}(t)\Delta}u^{\delta,R,N}(t)$, $t\in [0,T],$ for \eqref{wong-zakai}. It can be seen that $v^{\delta,R,N}$  satisfies
\begin{align}\label{eqn:twisted_nls_w_wong-zakai_approx}
dv^{\delta,R,N}=\bi\lambda e^{-\bi B^{\delta,R} (t)\Delta}\pi_{N}\left(\left|e^{\bi B^{\delta,R}(t)\Delta}v^{\delta,R,N}(t)\right|^2e^{\bi B^{\delta,R}(t)\Delta}v^{\delta,R,N}(t)\right)dt.
\end{align}
Following \cite{ostermann2018low} we can then construct a low-regularity integrator as follows. We begin by considering the integral form of \eqref{eqn:twisted_nls_w_wong-zakai_approx}%
\begin{align}\begin{split}\label{eqn:Duhamel_twisted}
&v^{\delta,R,N}(t_n+\tau)=v^{\delta,R,N}(t_n)\\
&+\bi\lambda \int_{t_n}^{t_n+\tau} e^{-\bi B^{\delta,R} (s)\Delta}\pi_{N}\left(\left|e^{\bi B^{\delta,R}(s)\Delta}v^{\delta,R,N}(s)\right|^2e^{\bi B^{\delta,R}(s)\Delta}v^{\delta,R,N}(s)\right)ds.\end{split}
\end{align}
In Fourier coordinates this can be written as follows, for $k\in\mathbb{Z},|k|\leq N$:
\begin{align*}
    (&\widehat{v^{\delta,R,N}})_{k}(t_n+\tau)=(\widehat{v^{\delta,R,N}})_{k}(t_n)+\sum_{\substack{k+k_1=k_2+k_3\\|k|,|k_1|,|k_2|,|k_3|\leq N}}\bi\lambda\int_{t_n}^{t_n+\tau}e^{\bi B^{\delta,R} (s)(k^2+k_1^2-k_2^2-k_3^2)} \\
&\quad \times \overline{(\widehat{v^{\delta,R,N}})_{k_1}(s)}(\widehat{v^{\delta,R,N}})_{k_2}(s)(\widehat{v^{\delta,R,N}})_{k_3}(s)ds, 
\end{align*}
which can equivalently be written as
\begin{align*}
    &(\widehat{v^{\delta,R,N}})_{k}(t_n+\tau)=(\widehat{v^{\delta,R,N}})_{k}(t_n)+\bi\lambda\!\!\!\!\sum_{\substack{k+k_1=k_2+k_3\\|k|,|k_1|,|k_2|,|k_3|\leq N}}e^{\bi B^{\delta,R} (t_n)(k^2+k_1^2-k_2^2-k_3^2)}\\
    &\underbrace{\int_{0}^{\tau} e^{\bi(B^{\delta,R} (t_n+s)-B^{\delta,R}(t_n))(k^2+k_1^2-k_2^2-k_3^2)}\overline{(\widehat{v^{\delta,R,N}})_{k_1}(s)}(\widehat{v^{\delta,R,N}})_{k_2}(s)(\widehat{v^{\delta,R,N}})_{k_3}(s)ds}_{=:\mathcal T_1}.
\end{align*}
On the surface $k+k_1=k_2+k_3$ we can simplify the exponent in the aforementioned construction using the ideas developed in \cite{ostermann2018low} namely,
\begin{align}\label{eqn:standard_approx_resonances}
    k^2+k_1^2-k_2^2-k_3^2&=2k_1^2+2k_2k_3-2k_1(k_2+k_3).
\end{align}
\eqref{eqn:standard_approx_resonances} was used in \cite{ostermann2018low} in the deterministic case to construct the first explicit low-regularity integrator for the NLSE. Following this construction in our case we have
\begin{align*}
    \mathcal T_1&\approx\int_{0}^{\tau} e^{2\bi k_1^2(B^{\delta,R}(t_n+s)-B^{\delta,R} (t_n))}\overline{(\widehat{v^{\delta,R,N}})_{k_1}(s)}(\widehat{v^{\delta,R,N}})_{k_2}(s)(\widehat{v^{\delta,R,N}})_{k_3}(s)ds\\
    &\approx\int_{0}^{\tau} e^{2\bi k_1^2(B^{\delta,R} (t_n+s)-B^{R,\delta }(t_n))}ds \ \overline{(\widehat{v^{\delta,R,N}})_{k_1}(t_n)}(\widehat{v^{\delta,R,N}})_{k_2}(t_n)(\widehat{v^{\delta,R,N}})_{k_3}(t_n).
\end{align*}
At this point we ``twist'' back to the coordinate $u^{\delta,R,N}(t)=e^{\bi B^{\delta,R}(t)\Delta}v^{\delta,R, N}(t)$ which leads to the following numerical method
{\small\begin{align}\begin{split}\label{eqn:construction_first_order_lri}  (&\widehat{u^{\delta,R,N}_{n+1}})_{k}=e^{-\bi(B^{\delta,R}(t_n+\tau)-B^{\delta,R}(t_n)) k^2}(\widehat{u_n^{\delta,R,N}})_{k}+e^{-\bi(B^{\delta,R}(t_n+\tau)-B^{\delta,R}(t_n)) k^2}\\
     &\quad\times \hspace{-0.5cm}\sum_{\substack{k+k_1=k_2+k_3\\|k|,|k_1|,|k_2|,|k_3|\leq N}}\hspace{-0.5cm}\bi\lambda\underbrace{\int_{0}^{\tau} e^{2\bi k_1^2(B^{\delta,R} (t_n+s)-B^{\delta,R} (t_n))}ds}_{=:I_{n,\tau}(k_1^2)} \ \overline{(\widehat{u_n^{\delta,R,N}})_{k_1}}(\widehat{u_{n}^{\delta,R,N}})_{k_2}(\widehat{u_n^{\delta,R,N}})_{k_3}.
     \end{split}
\end{align}}
 
In the case of the piecewise linear Wong--Zakai approximation \eqref{eqn:linear_WZ_approx} the remaining integrals can be computed exactly and can be expanded as follows
\begin{align*}
    e^{2\bi k_1^2B^{\delta,R} (t_n)}I_{n,\tau}(k_1^2)&=\sum_{\ell=0}^{\tau/\delta-1}\int_{t_n+\ell \delta}^{t_n+(\ell+1)\delta} e^{2\bi k_1^2B^{\delta,R} (s)}ds.
\end{align*}
On each of those subintervals the Wong--Zakai approximation is an affine function. Therefore the integrals can be computed exactly and are equal to
\begin{align*}
    \int_{a}^{b} e^{2\bi k_1^2B^{\delta,R} (s)}ds=e^{2\bi k_1^2B^{\delta,R}(a)}(b-a)\varphi_1(2\bi k_1^2(B^{\delta,R}(b)-B^{\delta,R}(a))),
\end{align*}
where $\varphi_1(z)=(e^z-1)/z$ for $z\neq 0$ and $\varphi_1(0)=1$. Thus we have
\begin{align*}
&e^{2\bi k_1^2B^{\delta,R}(t_n)}I_{n,\tau}(k_1^2)\\
&=\delta\sum_{\ell=0}^{\tau/\delta-1}e^{2\bi k_1^2 B^{\delta,R}(t_n+\ell\delta)}\varphi_1(2\bi k_1^2 (B^{\delta,R}(t_n+(\ell+1)\delta)-B^{\delta,R}(t_n+\ell\delta))).
\end{align*}
This expression together with \eqref{eqn:construction_first_order_lri} allows us to define the following resonance-based scheme for \eqref{model2} (SDLRI - stochastic dispersion low-regularity integrator):
\begin{align}\begin{split}\label{eqn:1st_order_lri}
    u_{n+1}^{\delta,R,N}&=e^{\bi (B^{\delta,R}(t_n+\delta)-B^{\delta,R}(t_n))\Delta}u_n^{\delta,R,N}\\
&\quad\quad+\bi\lambda e^{\bi (B^{\delta,R}(t_n+\delta)-B^{\delta,R}(t_n))\Delta}\pi_N\left(\left(u_n^{\delta,R,N}\right)^2I_{n,\tau}(-\Delta)\overline{u_n^{\delta,R,N}}\right).\end{split}
\end{align}

\begin{remark}
    Note in the above we exploited the piecewise linear nature of our Wong--Zakai approximation \eqref{eqn:linear_WZ_approx}. However, we note that the integrals $I_{n,\tau}(k_1^2)$ could also be computed with a Levin method \cite{levin1982procedures,iserles2024accelerated} in more general Wong--Zakai approximations. In particular, the integral
    \begin{align*}
 \int_{t_n+\ell\delta}^{t_n+(\ell+1)\delta}e^{2\bi k_1^2B^{\delta,R}(s)}ds
    \end{align*}
    can be written as
    \begin{align*} \int_{t_n+\ell\delta}^{t_n+(\ell+1)\delta}e^{2\bi k_1^2B^{\delta,R}(s)}ds=e^{2\bi k_1^2B_{R,\delta}^\delta(t_{n}+(\ell+1)\delta)}a(\delta t)-e^{2\bi k_1^2 B^{\delta,R}(t_{n}+\ell\delta)}a(0),
    \end{align*}
    where $a$ satisfies the following SDE
    \begin{align*}
    d a(s)+2\bi k_1^2a(s)\circ d \tilde B^{\delta,R}(s)=ds,
\end{align*}
where $\tilde B^{\delta,R}(s)= B^{\delta,R}(t_n+\ell \delta+s)- B^{\delta,R}(t_n+\ell \delta)$ and $s\in [0,\delta].$
This equation is less stiff than the original PDE and thus permits alternative numerical solution methodologies. In the interest of brevity we do not study this approach further here.
\end{remark}
\subsection{Pathwise convergence analysis}\label{path-subsec}
Having designed our resonance-based scheme \eqref{eqn:1st_order_lri} we now seek to study the convergence properties of the scheme, firstly with respect to the Wong--Zakai approximation \eqref{wong-zakai} in the pathwise sense and later in the mean-square sense. We follow the usual Lax equivalence paradigm.

\subsubsection{Stability}
Let us denote the scheme \eqref{eqn:1st_order_lri} by the nonlinear map $\Phi_{n,\tau}:u^{\delta,R,N}_n\mapsto u^{\delta,R,N}_{n+1}$, then we have the following stability estimate.
\begin{proposition}\label{prop:pathwise_stability}
    Let $\bs \in \mathbb N$, $\gamma\ge 0$ and $\bs+\gamma>\frac 12$.
There exists a strictly increasing function $C:\mathbb{R}_{\geq0}\times \mathbb{R}_{\geq0}\rightarrow\mathbb{R}_{\geq0}$ such that for any $\delta\in (0,1)$, $R\in [1,+\infty]$, $N\in \mathbb N^+\cup\{+\infty\}$, all $0<\tau=\frac {T}{K}$, $n\le K-1, n\in \mathbb{N}$, and any $z,w\in \pi_N L^2$, we have that
    \begin{align*}   \|\Phi_{n,\tau}(w)-\Phi_{n,\tau}(z)\|_{H^{\bs}}\leq \Big(1+\tau C(\|w\|_{H^{\bs+\gamma}},\|z\|_{H^{\bs+\gamma}})\Big)\|w-z\|_{H^\bs}, \text{a.s.}
    \end{align*}
\end{proposition}
\begin{proof} Firstly notice that for every $k_1\in\mathbb{Z}$ we have $|I_{n,\tau}(k_1^2)|\leq \tau,\,\,\text{a.s.}$, thus we have in Fourier coordinates
\begin{align*}
    \left|\widehat{\left(\Phi_{n,\tau}(z)\right)_k}-\widehat{\left(\Phi_{n,\tau}(w)\right)_k}\right|\leq |\hat{z}_k-\hat{w}_k|+\tau \sum_{\substack{k+k_1=k_2+k_3\\|k|,|k_1|,|k_2|,|k_3|\leq N}}\left|\overline{\hat{z}_{k_1}}\hat{z}_{k_2}\hat{z}_{k_3}-\overline{\hat{w}_{k_1}}\hat{w}_{k_2}\hat{w}_{k_3}\right|.
\end{align*}
Therefore, for $\bs+\gamma>1/2$, there is a  polynomial $C(\cdot,\cdot)>0$ of degree 2 such that
\begin{align*}
   \| \Phi_{n,\tau}(z)-\Phi_{n,\tau}(w)\|_{H^\bs}&\leq \|z-w\|_{H^{\bs}}+\tau  \|z-w\|_{H^{\bs}}C(\|w\|_{H^{\bs+\gamma}},\|z\|_{H^{\bs+\gamma}}), \,\,\text{a.s.}
\end{align*}
which immediately implies the desired result.
\end{proof}

\subsubsection{Local error}
 Let us denote by $\varphi_{t_a,t_b}(u_a)$ the exact flow of \eqref{wong-zakai} over the interval $t\in[t_{a},t_{b}]$ with the initial condition $u_a$ at $t=t_a$. 
 Note that if $u_a$ is $\mathcal F_{t_a}$-measurable, then $\varphi_{t_a,t_b}(u_a)$ is $\mathcal F_{t_b}$-measurable by the definition of the mild solution.
 Then we have the following local error estimate.
\begin{proposition}\label{prop:pathwise_local_error}
   Let $\bs \in \mathbb N$, $\gamma\in [0,1],$ and $\bs+\gamma>\frac 12$. There exists a strictly increasing function $M:\mathbb{R}_{\geq 0}\rightarrow \mathbb{R}_{\geq 0}$ such that for  any $\delta\in (0,1)$, $R\in [1,+\infty]$, $N\in \mathbb N^+\cup\{+\infty\}$, all $0<\tau=\frac {T}{K}$, $n\le K-1, n\in \mathbb{N}$, and any initial condition $z\in \pi_{N}L^2,$ we have that 
    \begin{align*}
&\|\varphi_{t_n,t_{n}+\tau}(z)-\Phi_{n,\tau}(z)\|_{H^\bs}\\
&\leq \left( \int_{0}^{\tau}|B^{\delta,R}(s+t_n)-B^{\delta,R}(t_n)|^{\gamma} ds+\tau^2\right)M\!\!\left(\sup_{t\in [0,\tau]}\|\varphi_{t_n,t_{n}+t}(z)\|_{H^{\bs+\gamma}}\right),\,\,\text{a.s.}
    \end{align*}
\end{proposition}
\begin{proof}
We aim to compare the mild form of the solution (after twisting back in \eqref{eqn:Duhamel_twisted}) and the numerical scheme \eqref{eqn:1st_order_lri}. The mild form of \eqref{wong-zakai} is:
\begin{align}\begin{split}\label{eqn:Duhamel_WZ}
    &\varphi_{t_n,t_{n}+s}(z)=e^{\bi (B^{\delta,R}(t_n+\tau)-B^{\delta,R} (t_n))\Delta}z\\
    &\quad+\bi\lambda e^{\bi B^{\delta,R}(t_n+\tau)\Delta}\int_{t_n}^{t_n+\tau}e^{-\bi B^{\delta,R}(s)\Delta}\pi_{N}\left(\left|\varphi_{t_n,t_{n}+s}(z)\right|^2\varphi_{t_n,t_{n}+s}(z)\right) ds.\end{split}
\end{align}
This immediately implies, noting that $e^{\bi B^{\delta,R}(\cdot)\Delta}$ is an isometry on $H^{\bs}$ and using standard bilinear estimates, for any $s\in [0,\tau],$
\begin{align*}
\|\varphi_{t_n,t_{n}+s}(z)-e^{\bi (B^{\delta,R}(t_n+s)-B^{\delta,R}(t_n))\Delta}z\|_{H^\bs}\leq \tau |\lambda|\sup_{t\in[0,\tau]}\|\varphi_{t_n,t_n+t}(z)\|_{H^{\bs+\gamma}}^3.
\end{align*}
Therefore we find by iteration in \eqref{eqn:Duhamel_WZ}
\begin{align*}
&\varphi_{t_n,t_{n}+\tau}(z)=e^{\bi (B^{\delta,R}(t_n+\tau)-B^{R,\delta }(t_n))\Delta}z+\bi\lambda e^{\bi B^{R,\delta }(t_n+\tau)\Delta}\int_{t_n}^{t_n+\tau}e^{-\bi B^{R,\delta } (s)\Delta}\\ &\quad\times \pi_{N}\!\!\left(\left|e^{\bi (B^{R,\delta }(t_n+\tau)-B^{R,\delta }(t_n))\Delta}z\right|^2e^{\bi (B^{R,\delta }(t_n+\tau)-B^{R,\delta } (t_n))\Delta}z\right) ds+\mathcal{R}(z),
\end{align*}
where
\begin{align}\label{eqn:estimate_on_R}
    \|\mathcal{R}(z)\|_{H^\bs}\lesssim \tau^2\sup_{t\in[0,\tau]}\|\varphi_{t_n,t_n+t}(z)\|_{H^{\bs+\gamma}}^5.
\end{align}
 We can now compare $\varphi_{t_n,t_{n}+\tau}(z),\Phi_{n,\tau}(z)$ in Fourier coordinates:
\begin{align}\begin{split}\label{eqn:local_error_Fourier}
    &\widehat{(\varphi_{t_n,t_{n}+\tau}(z))}_{k}-\widehat{(\Phi_{n,\tau}(z))}_{k}-\widehat{\mathcal{R}(z)}_k\\
    &=\bi\lambda\!\!\!\!\sum_{\substack{k+k_1=k_2+k_3\\|k|,|k_1|,|k_2|,|k_3|\leq N}}e^{-\bi (B^{\delta,R} (t_n+\tau)-B^{\delta,R} (t_n))k^2}\overline{\hat{z}}_{k_1}\hat{z}_{k_2}\hat{z}_{k_3}\\
    &\quad
    \times\left(\int_{0}^{\tau} e^{\bi(B^{\delta,R} (t_n+s)- B^{\delta,R}(t_n))(k^2+k_1^2-k_2^2-k_3^2)}-e^{\bi(B^{\delta,R} (t_n+s)-B^{\delta,R}(t_n))2k_1^2}ds\right).
\end{split}\end{align}
Now we note that for $k+k_1=k_2+k_3$, we have (cf. \eqref{eqn:standard_approx_resonances})
\begin{align}\nonumber
    &\left|\int_{0}^{\tau} e^{\bi(B^{\delta,R}(t_n+s)-B^{\delta,R}(t_n))(k^2+k_1^2-k_2^2-k_3^2)}-e^{\bi (B^{\delta,R}(t_n+s)-B^{\delta,R}(t_n))2k_1^2}ds\right|\\\nonumber
    &\quad\quad\quad\quad\quad=\left|\int_{0}^{\tau} e^{\bi(B^{\delta,R} (t_n+s)-B^{\delta,R}(t_n))2(k_2k_3-k_1k_2-k_1k_3)}-1\,ds\right|\\\label{eqn:convolutional_kernel_estimate}
    &\quad\quad\quad\quad\quad\leq \int_{0}^\tau|B^{\delta,R} (t_n+s)-B^{\delta,R}(t_n)| ds 2(|k_2||k_3|+|k_1||k_2|+|k_1||k_3|)
\end{align}
where we used the fact that $|e^{\bi x}-1|\leq |x|$ for all $x\in\mathbb{R}$. Combining \eqref{eqn:convolutional_kernel_estimate} with \eqref{eqn:local_error_Fourier}, as well as a standard interpolation argument, we find
\begin{align*}
    \left\|\varphi_{t_n,t_{n}+\tau}(z)-\Phi_{n,\tau}(z)-\mathcal{R}(z)\right\|_{H^\bs}\lesssim \int_{0}^\tau|B^{\delta,R} (t_n+s)-B^{\delta,R}(t_n)|^{\gamma} ds \|z\|_{H^{\bs+\gamma}}^3,
\end{align*}
where we used that if for two functions $w,v\in H^{\bs+\gamma}$ we have $|\hat{w}_k|\leq |\hat{v}_k|, \forall k\in\mathbb{Z}$, then $\|w\|_{H^{\bs+\gamma}}\leq \|v\|_{H^{\bs+\gamma}}$. Using this, together with \eqref{eqn:estimate_on_R}, we complete the proof.
\end{proof}

\subsubsection{Global error}
In this part, we are in a position to present the global convergence result of the proposed scheme. To this end, we first give an a priori estimate of the numerical solution which is the key to establish the mean square convergence of first order for scheme \eqref{eqn:1st_order_lri}.

\begin{lm}\label{lm-pri-hs-est}
Let $T>0$, $\bs,\gamma \in \mathbb N,$  $\bs+\gamma\in \mathbb N^+$ and $u_0\in L^2$.
 There exists a constant $\epsilon>0$ such that if $\|u_0\|_{\dot H^{\bs+\gamma}}\le \epsilon,$
then  for  \eqref{wong-zakai} with $R\in [1,+\infty]$, $\delta\in(0,\tau]$ with $0<\tau=\frac {T}{K}$ and $n\le K$, and $N\in \mathbb N^+\cup\{+\infty\}$ 
satisfying that $ 
 R\epsilon$ and $\tau$ are sufficient small, 
it holds that 
\begin{align*}\|u^{\delta,R,N}_n\|_{H^{\bs+\gamma}}\lesssim \epsilon+\|u_0\|_{L^2},\,\, \text{a.s.} 
\end{align*}
\end{lm}

\begin{proof}
According to Theorem \ref{small-wel}, we can take $\epsilon$ small such that if $\|u_0\|_{\dot H^{\bs+\gamma}}\le \epsilon$, then
\begin{align*}
  \sup_{t\in [0,T]}\|u^{\delta,R,N}(t)\|_{\dot H^{\bs+\gamma}}\leq  2\epsilon,\,\,\text{a.s.}
\end{align*}
Define the stopping step $K_1=\inf\{0<k\le K|\|u_{k}^{\delta,R,N}\|_{\dot H^{\bs+\gamma}}>2\epsilon\}$. 
Since our scheme does not necessarily conserve mass (i.e. the $L^2$ norm) we use a bootstrap argument and start with an $L^2$-estimate as follows.

\textbf{Step 1:}
By the triangle inequality,
we have that 
\begin{align*}
\|u_{n+1}^{\delta,R,N}-u^{\delta,R,N}(t_{n+1})\|_{L^2}
&\le \|\phi_{t_{n},t_{n+1}}(u^{\delta,R,N}(t_n))-\Phi_{n,\tau}(u^{\delta,R,N}(t_n))\|_{L^2}\\&\quad+\|\Phi_{n,\tau}(u^{\delta,R,N}(t_n))-\Phi_{n,\tau}(u_{n}^{\delta,R,N})\|_{L^2}.
\end{align*}
On the one hand, repeating the procedures from the proof of  Proposition \ref{prop:pathwise_local_error}, using
\eqref{GN-ine} and the condition that $\bs+\gamma\in \mathbb N^+$, for $n\le K_1-1,$ it holds that 
\begin{align}\begin{split}\label{err-local}    &\|\phi_{t_{n},t_{n+1}}(u^{\delta,R,N}(t_n))-\Phi_{n,\tau}(u^{\delta,R,N}(t_n))\|_{L^2}\\
    &\lesssim \int_0^{\tau}|B^{\delta,R}(s+t_n)-B^{\delta,R}(t_n)|ds\sup_{t\in [0,\tau]}\|\phi_{t_n,t_n+t}(u^{\delta,R,N}(t_n))\|_{\dot H^{\bs+\gamma}}
\\
&\quad\times \|u^{\delta,R,N}(t_n)\|_{L^2}^2+\tau^2\sup_{t\in [0,\tau]}\|\phi_{t_n,t_n+t}(u^{\delta,R,N}(t_n))\|_{ H^{\bs+\gamma}}^5\\
&\lesssim  \int_0^{\tau}|B^{\delta,R}(s+t_n)-B^{\delta,R}(t_n)| ds\Big(\|u_0\|_{L^2}^2\epsilon
+\tau^2(\|u_0\|_{L^2}+\epsilon)^5\Big).
\end{split}\end{align}
On the other hand, similarly to the proof of   Proposition \ref{prop:pathwise_stability}, using the mass conservation law of $u^{\delta,R,N}(\cdot)$, for $n\le K_1-1$,
{\small
\begin{align}\begin{split}\label{err-gronwall}
&\|\Phi_{n,\tau}(u^{\delta,R,N}(t_n))-\Phi_{n,\tau}(u^{\delta,N,n})\|_{L^2}-\|u^{\delta,R,N}(t_n)-u^{\delta,R,N}
_n\|_{L^2}\\
&\lesssim \tau \|u^{\delta,R,N}(t_n)-u_n^{\delta,R,N}\|_{L^2}\\
&\quad\times (\|u^{\delta,R,N}(t_n)\|_{\dot H^{\bs+\gamma}}\|u_{n}^{\delta,R,N}(t_n)\|_{L^2}+\|u_{n}^{\delta,R,N}\|_{\dot H^{\bs+\gamma}}\|u_{n}^{\delta,R,N}\|_{L^2})\\
 &\lesssim  \tau \|u^{\delta,R,N}(t_n)-u_{n}^{\delta,R,N}\|_{L^2} \Big(\epsilon \|u_0\|_{L^2}+\epsilon( \|u_0\|_{L^2}+\|u_{}^{\delta,R,N}(t_n)-u_{n}^{\delta.R,N}\|_{L^2})\Big)\\
 &\lesssim  \epsilon\tau \|u^{\delta,R,N}(t_n)-u_{n}^{\delta,R,N}\|_{L^2}  \|u_0\|_{L^2}
+\epsilon \tau \|u^{\delta,R,N}(t_n)-u_{n}^{\delta,R,N}\|_{L^2}^2.
\end{split}\end{align}}Combining \eqref{err-gronwall}
 and \eqref{err-local} together, and using the fact that $|B^{\delta,R}(t)|\le R\sqrt{T}$, we have that \begin{align*}
&\|u_{}^{\delta,R,N}(t_{n+1})-u_{n+1}^{\delta,R,N}\|_{L^2} -\|u_{}^{\delta,R,N}(t_{n})-u^{\delta,R,N}_n\|_{L^2} \\
   &\lesssim \|u_{}^{\delta,R,N}(t_{n})-u^{\delta,R,N}_n\|_{L^2}\tau \epsilon \|u_0\|_{L^2}
  +\epsilon\tau \|u^{\delta,R,N}(t_{n})-u_{n}^{\delta,R,N}\|_{L^2}^2\\
&\quad+\Big(\int_0^{\tau}|B^{\delta,R}(s+t_n)-B^{\delta,R}(t_n)|ds\|u_0\|_{L^2}^2 \epsilon
+\tau^2(\|u_0\|_{L^2}+\epsilon)^5\Big)\\
 &\lesssim \|u^{\delta,R,N}(t_{n})-u_{n}^{\delta,R,N}\|_{L^2}\tau \epsilon \|u_0\|_{L^2}
   +\epsilon\tau \|u^{\delta,R,N}(t_{n})-u_n^{\delta,R,N}\|_{L^2}^2\\
&\quad+ R\sqrt{T}\tau \|u_0\|_{L^2}^2\epsilon
+\tau^2(\|u_0\|_{L^2}+\epsilon)^5.
\end{align*}

\textbf{Step 2:}
Define $K_2=\inf\limits_{}\{k\le K_1|\|u^{\delta,R,N}(t_{k})-u_{k}^{\delta,R,N}\|_{L^2}>2\epsilon+2\|u_0\|_{L^2}\}.$
Then before $n\le K_2-1,$ there exists $C_1>0$ such that
\begin{align*}  
&\|u^{\delta,R,N}(t_{n+1})-u_{n+1}^{\delta,R,N}\|_{L^2}  \\\nonumber
&\le  C_1e^{C_1T \epsilon\|u_0\|_{L^2}} \Big(\sum_{i=0}^{n} \tau\epsilon \|u^{\delta,R,N}(t_{i})-u_{i}^{\delta,R,N}\|_{L^2}^2+ RT^{\frac 32}\|u_0\|_{L^2}^2 \epsilon
+T\tau(\|u_0\|_{L^2}+\epsilon)^5 \Big).
\end{align*}
Now, choosing $R\epsilon$ and $\tau$ sufficient small such that 
\begin{align}\begin{split}\label{cond-l2-pri}
e^{C_1T\epsilon\|u_0\|_{L^2}} C_1 T\epsilon(\epsilon+\|u_0\|_{L^2})  &\le \frac 1 4,\\
 C_1Te^{C_1T\epsilon\|u_0\|_{L^2}}\Big[  T^{\frac {1} 2}\|u_0\|_{L^2}^2\epsilon R
+\tau(\|u_0\|_{L^2}+\epsilon)^5\Big]&\le \epsilon+\|u_0\|_{L^2},
\end{split}\end{align}the bootstrap argument yields that 
 for all $n\le K_1=K_2,$
\begin{align}\label{final-l2-bound}    \|u^{\delta,R,N}(t_{n})-u_n^{\delta,R,N}\|_{L^2} \le 2\epsilon +2\|u_0\|_{L^2}.
\end{align}

Next we apply the bootstrap  argument again to the homogeneous Sobolev norm. By  \eqref{GN-ine} and the unitary property of $e^{\bi \Delta B^{\delta,R}(\cdot)}$, there exists a constant $C_{sob,1}'>0$ such that for $n\le K_1-1,$
\begin{align}\label{GN-iterative1}
&\|u_{n+1}^{\delta,R,N}\|_{\dot H^{\bs+\gamma}}\le \|u_n^{\delta,R,N}\|_{\dot H^{\bs+\gamma}}+C_{sob,1}'C_{GN}^2|\lambda|\tau \|u_n^{\delta,R,N}\|_{\dot H^{\bs+\gamma}}^{1+\frac 1{\bs+\gamma}}\|u_n^{\delta,R,N}\|_{L^2}^{2-\frac 1{\bs+\gamma}}.
\end{align}

As a consequence of \eqref{final-l2-bound}-\eqref{GN-iterative1} and the conservation law \eqref{pri-mass-con}, we obtain that for all $n\le K_1-1,$
\begin{align*}\|u_{n+1}^{\delta,R,N}\|_{\dot H^{\bs+\gamma}} 
&\le \|u_{0}\|_{\dot H^{\bs+\gamma}}
+\sum_{i=0}^{n}C_{sob,1}'C_{GN}^2|\lambda|\tau \|u_i^{\delta,R,N}\|_{\dot H^{\bs+\gamma}}^{1+\frac 1{\bs+\gamma}}\|u_i^{\delta,R,N}\|_{L^2}^{2-\frac 1{\bs+\gamma}}\\
&\le \epsilon+ C_{sob,1}'C_{GN}^2|\lambda| T(2\epsilon)^{1+\frac 1{\bs+\gamma}}(2\epsilon+3\|u_0\|_{L^2})^{2-\frac 1 {\bs+\gamma}}.
\end{align*}
By further requiring $\epsilon$ small enough such that $$C_{sob,1}'C_{GN}^2|\lambda| T 2^{1+\frac 1{\bs+\gamma} }\epsilon ^{\frac 1{\bs+\gamma}}(2\epsilon+3\|u_0\|_{L^2})^{2-\frac 1 {\bs+\gamma}}\le 1,$$
we have obtained  $\|u_n^{\delta,R,N}\|_{\dot H^{\bs+\gamma}}\le 2\epsilon$ for all $n\le K=K_1$ via the bootstrap argument. This, together with \eqref{final-l2-bound}, completes the proof.

\end{proof}

\begin{rk}\label{rk-3.1}
If in addition we assume that $\|u_0\|_{H^{\bs+\gamma}}\le \epsilon$ for a small $\epsilon$ such that $2^3C_{sob,1}'C_{GN}^2|\lambda|T\epsilon^2\le 1$,
then by using similar arguments as in the proof of Theorem~\ref{small-wel},
one can obtain that 
\begin{align*}
\sup_{n\le K,N\in\mathbb N^+}\sup_{R\in [1,+\infty]} \|u_n^{\delta,R,N}\|_{ H^{\bs+\gamma}}\lesssim \epsilon. 
\end{align*}
Note that under this stronger smallness condition, there is no need to introduce a finite  truncation number $R$.  
\end{rk}

\begin{theorem}[Pathwise convergence in the Wong--Zakai approximation]\label{thm:global_error_pathwise_convergence_wong-zakai}Let $T>0$, $\bs,\gamma \in \mathbb N, \bs+\gamma\in \mathbb N^+,$   and $u_0\in L^2$.
 There exists a constant  $\epsilon>0$ such that if $\|u_0\|_{\dot H^{\bs+\gamma}}\le \epsilon,$
then  for   \eqref{wong-zakai} with $R\in [1,+\infty]$, $\delta\in(0,\tau]$ with $0<\tau=\frac {T}{K}$ and $n\le K$, and $N\in \mathbb N^+\cup\{+\infty\}$ 
satisfying that $ 
 R\epsilon$ and $\tau$ are sufficient small,
we have 
    \begin{align*}
&\left\|u^{\delta,R,N}(t_n)-u^{\delta,R,N}_n\right\|_{H^{\bs}}\lesssim  \tau+ \sum_{i=0}^{n-1}\int_0^{\tau}|B^{\delta,R}(t+t_i)-B^{\delta,R}(t_i)|dt,\,\, \text{a.s.} 
    \end{align*}
\end{theorem}
\begin{proof} Let us write $E_n$ for the global error at a given time $t_n\leq T$, i.e.
\begin{align*}
    E_n:=\left\|u^{\delta,R,N}(t_n)-u^{\delta,R,N}_n\right\|_{H^{s}}
\end{align*}
Then we have that for $n\le K-1,$
\begin{align*}   E_{n+1}&=\left\|\varphi_{t_{n},t_{n}+\tau}(u^{\delta,R,N}(t_{n}))-\Phi_{n,\tau}(u^{\delta,R,N}_{n})\right\|_{H^{s}}\\
        &\leq \left\|\varphi_{t_{n},t_{n}+\tau}(u^{\delta,R,N}(t_{n}))-\Phi_{n,\tau}(u^{\delta,R,N}(t_{n}))\right\|_{H^{s}}\\
        &\quad+\left\|\Phi_{n,\tau}(u^{\delta,R,N}(t_{n}))-\Phi_{n,\tau}(u^{\delta,R,N}_{n})\right\|_{H^{s}}.
\end{align*}
Using Propositions~\ref{prop:pathwise_stability} \& \ref{prop:pathwise_stability} and Lemma \ref{lm-pri-hs-est} we deduce that $\exists M,C>0$ such that 
\begin{align*}
    E_{n+1}&\leq \left(\int_{0}^{\tau}|B^{\delta,R}(s+t_n)-B^{\delta,R}(t_n)| ds+\tau^2\right) M+e^{\tau C}E_{n}\\
        &\leq \left(\tau+ \int_0^{\tau}|B^{\delta,R}(s+t_n)-B^{\delta,R}(t_n)|ds \right)M+e^{\tau C}E_{n},
\end{align*}
for some constants $M,C>0$ which only depend on $\|u_0\|_{ H^{\bs+\gamma}}$  polynomially. Moreover, clearly $E_0=0$. Thus by induction
\begin{align}\label{eqn:global_error_estimate_induction_step}
    E_{n+1}&\leq e^{TC}M \left(T\tau+ \sum_{i=0}^n\int_0^{\tau}|B^{\delta,R}(t+t_i)-B^{\delta,R}(t_i)|dt \right),
\end{align}
which completes the proof.
\end{proof}


Let the smallness assumption of Theorem \ref{thm:global_error_pathwise_convergence_wong-zakai} hold with $\gamma=1$.
 By taking expectation to the estimates in Proposition 
\ref{prop:approx_prop_wong_zakai_periodic} and Theorem \ref{thm:global_error_pathwise_convergence_wong-zakai}, and using \eqref{hold-est} and Lemma \ref{lm-improv-order}, we have that for 
   $R\in [\sqrt{8|\ln(\delta)|},+\infty]$, $\delta\in(0,\tau]$,  and $N\in \mathbb N^+\cup\{+\infty\}$ 
satisfying that $ 
 R\epsilon$ and $\tau$ are sufficient small, 
\begin{align*}
 \sup_{n\le K}\E[\|u^{\delta,R,N}(t_n)-u(t_n)\|_{H^{\bs}}]\lesssim (N^{-\gamma}+\delta^{\frac 12\min(1,\frac \gamma 2)}+\tau^{\frac 1 2}).  
\end{align*}
Compared with the existing result \cite{cui2017stochastic,cohen2017exponential} where they need two additional bounded derivatives on $u_0$ to derive $\frac 12$-convergence order w.r.t. $\tau$, we only need one additional regularity condition to achieve the same convergence order.

\subsection{Improved mean-square error analysis}

In this part, we show that the mean-square convergence order w.r.t. $\tau$ of the low regularity integrator towards to the Wong--Zakai approximation \eqref{wong-zakai} can be improved. The key ingredients of the proof are using the conditional expectation and the structure of the resonance-based integrator.

\begin{theorem}\label{tm-improve}
Let $T>0$, $\bs\in \mathbb N,\gamma \in \mathbb N^+,$   and $u_0\in L^2$.
 There exists a constant  $\epsilon>0$ such that if $\|u_0\|_{\dot H^{\bs+\gamma}}\le \epsilon,$
then for 
\eqref{wong-zakai} with $R\in [\sqrt{8|\ln(\delta)|},+\infty]$, $\delta\in(0,\tau]$ with $0<\tau=\frac {T}{K}$ and  $n\le K$, and $N\in \mathbb N^+\cup\{+\infty\}$ 
satisfying that $ 
 R\epsilon$ and $\tau$ are sufficient small,
it holds that  
    \begin{align*}
&\sup_{n\le K}\left\|u^{\delta,R,N}(t_n)-u^{\delta,R,N}_n\right\|_{L^2(\Omega;H^{\bs})} \lesssim  \tau^{\max(\frac {\gamma+1}{4},1)}. 
    \end{align*}
\end{theorem}

\begin{proof}
We follow the same notations as in the proof of Theorem  \ref{thm:global_error_pathwise_convergence_wong-zakai}. It can be seen that for $n\le K-1,$ 
 \begin{align}\nonumber
 \mathbb E [E_{n+1}^2]     &=\mathbb E \left\|\varphi_{t_{n},t_{n}+\tau}(u^{\delta,R,N}(t_{n}))-\Phi_{n,\tau}(u^{\delta,R,N}_{n})\right\|_{H^{\bs}}^2\\\nonumber
        &= \mathbb E \left\|\varphi_{t_{n},t_{n}+\tau}(u^{\delta,R,N}(t_{n}))-\Phi_{n,\tau}(u^{\delta,R,N}(t_{n}))\right\|_{H^{\bs}}^2\\\nonumber
&\quad+\mathbb E \left\|\Phi_{n,\tau}(u^{\delta,R,N}(t_{n}))-\Phi_{n,\tau}(u^{\delta,R,N}_n)\right\|_{H^{\bs}}^2\\
\nonumber &\quad+2\mathbb E \<\varphi_{t_{n},t_{n}+\tau}(u^{\delta,R,N}(t_{n}))-\Phi_{n,\tau}(u^{\delta,R,N}(t_{n})), \\\nonumber
&\quad\quad \quad  \Phi_{n,\tau}(u^{\delta,R,N}(t_{n}))-\Phi_{n,\tau}(u^{\delta,R,N}_{n})\>_{H^{\bs}}\\\label{err-formula}
&=: Er_{1}+Er_{2}+Er_{3}.
\end{align}
According to Propositions \ref{prop:pathwise_stability} and \ref{prop:pathwise_local_error}, using Theorem \ref{small-wel} and Lemma \ref{lm-pri-hs-est}, as well as \eqref{hold-est}, we have that there exists $c_0>0$ such that  
\begin{align}\nonumber
Er_1+Er_2&\le (1+c_0\tau)\E[E_n^2]
+c_0 \Big[\tau^4+\E \Big(\int_0^{\tau} |B^{\delta,R}(s+t_n)-B^{\delta,R}(t_n)|ds\Big)^2\Big]\\\label{err-er12}
&\le (1+c_0\tau)\E[E_n^2]
+c_0\tau^3.
\end{align}

It remains to bound $Er_3$.
Denote $z=v^{\delta,R,N}(t_{n}), y=v^{\delta,R,N}_n.$ Note that they are $\mathcal F_{t_{n}}$-measurable. Denote 
\begin{align*}
Er_{4,k}&=\bi\lambda\sum_{\substack{k+k_1=k_2+k_3\\|k|,|k_1|,|k_2|,|k_3|\leq N}}e^{i B^{\delta,R} (t_n)(k^2+k_1^2-k_2^2-k_3^2)}\overline{\hat{z}}_{k_1}\hat{z}_{k_2}\hat{z}_{k_3}\\
&\quad\times\left(\int_{0}^{\tau} e^{i(B^{\delta,R} (t_n+s)-B^{\delta,R}(t_n))(k^2+k_1^2-k_2^2-k_3^2)}-e^{i(B^{\delta,R} (t_n+s)-B^{\delta,R}(t_n))2k_1^2}ds\right).
\end{align*}
By the Parserval identity, we can compute the inner product with the twist variable, i.e.,
\begin{align*}
  &Er_3=2\E \Big\<{Er_{4,k}} +\widehat {\mathcal R}(z)_k ,\\
  &\quad (\widehat z)_k+\bi\lambda \sum_{k+k_1=k_2+k_3}e^{i B^{\delta,R}(t_n)(k^2+k_1^2-k_2^2-k_3^2)}I_{n,\tau} \overline{(\widehat z)_{k_1}}(\widehat z)_{k_2}(\widehat z)_{k_3}\\
  &-(\widehat y)_k+\bi\lambda \sum_{k+k_1=k_2+k_3}e^{i B^{\delta,R}(t_n)(k^2+k_1^2-k_2^2-k_3^2)}I_{n,\tau} \overline{(\widehat y)_{k_1}}(\widehat y)_{k_2}(\widehat y)_{k_3}\Big\>_{H^\bs_k}.
\end{align*}
Here $H_k^{\bs}$ denotes the Sobolev space in Fourier coordinates, and  $\mathcal R(z)$ is the remainder term defined in \eqref{eqn:estimate_on_R}.
We decompose $Er_3$ as 
\begin{align*}
  Er_3&=2 \mathbb E \<Er_{4,k}, \widehat{z}_k-\widehat y_k \>_{H_k^{\bs}}\\
  &+2 \mathbb E \<Er_{4,k}, \bi\lambda \sum_{k+k_1=k_2+k_3}e^{i B^{\delta,R}(t_n)(k^2+k_1^2-k_2^2-k_3^2)}I_{n,\tau} \\
  &\quad \times [\overline{(\widehat z)_{k_1}}(\widehat z)_{k_2}(\widehat z)_{k_3}-\overline{(\widehat y)_{k_1}}(\widehat y)_{k_2}(\widehat y)_{k_3}] \>_{H_k^\bs}\\
  &+2 \mathbb E \<{\mathcal R}(z)_k, \widehat{z}_k-\widehat y_k \>_{H_k^{\bs}}\\
  &+2 \mathbb E \<{\mathcal R}(z)_k, \bi\lambda \sum_{k+k_1=k_2+k_3}e^{i B^{\delta,R}(t_n)(k^2+k_1^2-k_2^2-k_3^2)}I_{n,\tau}\\
  &\quad \times[\overline{(\widehat z)_{k_1}}(\widehat z)_{k_2}(\widehat z)_{k_3}-\overline{(\widehat y)_{k_1}}(\widehat y)_{k_2}(\widehat y)_{k_3}] \>_{H_k^\bs}\\
  &:=J_1+J_2+J_3+J_4.
\end{align*}
By \eqref{eqn:estimate_on_R} and \eqref{eqn:convolutional_kernel_estimate}, as well as H\"older's and Young's  inequalities,
there exist constants $c_j>0,j=1,\dots,4,$ such that
\begin{align*}
J_3&\le c_1\tau^2 \mathbb E [\sup_{t\in [0,\tau]}\|u^{\delta,R,N}(t_{n}+t)\|^5_{H^{\bs+1}}E_{n}]\\
&\le c_2\tau \mathbb E E_{n}^2+c_2\tau^3\mathbb E [\sup_{t\in [0,\tau]}\|u^{\delta,R,N}(t_{n}+t)\|^{10}_{H^{\bs+1}}],\\
\text{and}\ \ J_4&\le c_3 \tau^3 \mathbb E [\sup_{t\in [0,\tau]}(\|u^{\delta,R,N}(t_{n}+t)\|^8_{H^{\bs+1}}+\|u^{\delta,R,N}_{n}\|^8_{H^{\bs+1}})E_{n}]\\
&\le  c_4\tau \mathbb E E_{n}^2+c_4\tau^3\mathbb E [\sup_{t\in [0,\tau]}( \|u^{\delta,R,N}(t_{n}+t)\|^{16}_{H^{\bs+1}}+\|u^{\delta,R,N}_{n}\|^{16}_{H^{\bs+1}})].
\end{align*}
For the term $J_2,$  using \eqref{hold-est} and Proposition \ref{prop:pathwise_local_error}, by H\"older's inequality and Young's inequality, there exists constants $c_5,c_6>0$ such that
\begin{align*}
  J_2&\le c_5 \tau \mathbb E \Big[\int_0^\tau |B^{\delta,R}(t_n+s)-B^{\delta,R}(t_n)|^2ds 
(1+\|z\|_{H^{\bs+1}}^5+\|y\|_{H^{\bs+1}}^5)E_{n}\Big]\\
&\le c_6\tau^5\E (1+\|z\|_{H^{\bs+1}}^{10}+\|y\|_{H^{\bs+1}}^{10})+c_6\tau \mathbb E[E_{n}^2].
\end{align*}
Finally, we estimate the term $J_1$. Using the conditional expectation and the independent increments  of Brownian motion, we have that
\begin{align*}
  J_1&=2\mathbb E\< \mathbb E [Er_{4,k}|\mathcal F_{t_{n}}], \widehat{z}_k-\widehat y_k\>_{H_k^\bs}. 
\end{align*}
Notice that
\begin{align*}
    &\mathbb E [Er_{4,k}|\mathcal F_{t_n}]\\
&=\bi\lambda\!\!\!\!\sum_{\substack{k+k_1=k_2+k_3\\|k|,|k_1|,|k_2|,|k_3|\leq N}}e^{i B^{\delta,R} (t_n)(k^2+k_1^2-k_2^2-k_3^2)}\overline{\hat{z}}_{k_1}\hat{z}_{k_2}\hat{z}_{k_3}\\ &\quad\quad\times\left(\int_{0}^{\tau} \mathbb E e^{i(B^{\delta,R} (t_n+s)-B^{\delta,R}(t_n))(k^2+k_1^2-k_2^2-k_3^2)}-\mathbb E e^{i(B^{\delta,R} (t_n+s)-B^{\delta,R}(t_n))2k_1^2}ds\right).
\end{align*}
The definition of $B^{\delta,R}$ and the properties of the Brownian motion, together with \eqref{eqn:standard_approx_resonances} and \eqref{err-trun} with $R\ge \sqrt{8|\ln(\delta)|}$, yield that for some $c_7>0, ${\small 
\begin{equation}\label{err-estimate1}
\begin{split}
     &\Big|\int_{0}^{\tau} \mathbb E e^{i(B^{\delta,R}(t_n+s)-B^{\delta,R}(t_n))(k^2+k_1^2-k_2^2-k_3^2)}-\mathbb E e^{i(B^{\delta,R} (t_n+s)-B^{\delta,R}(t_n))2k_1^2}ds\Big|\\
     &\le \int_0^{\tau}
     e^{-\frac 12[\frac {(s-[s]\delta)^2}{\delta}+[s]\delta](k^2+k_1^2-k_2^2-k_3^2)^2}-e^{-\frac 12[\frac {(s-[s]\delta)^2}{\delta}+[s]\delta](2k_1^2)^2} ds\\
&+c_7\tau\delta^2(k^2+k_1^2+k_2^2+k_3^2)\\
     &\le c_7\int_0^{\tau}e^{-\frac 12[\frac {(s-[s]\delta)^2}{\delta}+[s]\delta](2k_1^2)^2}  \Big(e^{-\frac 12[\frac {(s-[s]\delta)^2}{\delta}+[s]\delta][(2k_1^2-2k_2k_3-2k_1(k_2+k_3))^2-(2k_1^2)^2]}-1\Big)ds\\
&+c_7\tau\delta^2 [2k_1^2+2|k_2k_3|+2|k_1k_3|+|2k_2k_3|]\\
     &\le c_7\tau^2 [2k_1^2+2|k_2k_3|+2|k_1k_3|+|2k_2k_3|][2|k_2k_3|+2|k_1k_3|+|2k_2k_3|].
\end{split}
\end{equation}}Applying the interpolation arguments to \eqref{eqn:convolutional_kernel_estimate} and \eqref{err-estimate1},
 we have that for any $\theta\in [0,1],$ there exists $c_8,c_9>0$ such that
\begin{align}\nonumber
&\Big|\int_{0}^{\tau} \mathbb E e^{i(B^{\delta,R} (t_n+s)-B^{\delta,R}(t_n))(k^2+k_1^2-k_2^2-k_3^2)}-\mathbb E e^{i(B^{\delta,R} (t_n+s)-B^{\delta,R}(t_n))2k_1^2}ds\Big|\\\nonumber
     &\le c_8 \tau^{\frac 32\theta+2-2\theta}[|k_2k_3|+|k_1k_2|+|k_1k_3|][2k_1^2+2|k_2k_3|+2|k_1k_3|+|2k_2k_3|]^{1-\theta}\\\label{mean-square}
    &\le c_9\tau^{2-\frac 1 2\theta }
    (|k_2k_3|^{2-\theta}+|k_2k_1|^{2-\theta}+|k_1k_3|^{2-\theta} \\\nonumber
    &+|k_2k_3||k_1|^{2-2\theta}+|k_1|^{3-2\theta}|k_2|+|k_1|^{3-2\theta}|k_3|).
\end{align}
According to \eqref{mean-square} and H\"older's inequality, there exists $c_{10}>0$ such that
\begin{align*}
    J_1
    &\le c_{10}\tau^{3-\theta}\E[(1+\|z\|^6_{H^{\bs+3-2\theta}}+\|y\|^6_{H^{\bs+3-2\theta}})]+c_{10}\tau \mathbb E [E_{n}^2]. 
\end{align*}
Combining the estimates of $J_1$-$J_4$, applying the Gronwall's inequality and taking the square root, we obtain
\begin{align*}
    \sup_{n\le K}\sqrt{\E[E_{n}^2]}\lesssim \tau^{1-\frac \theta 2}
\end{align*}
under an additional $H^{\bs+3-2\theta}$ regularity condition on the initial data.
Taking $\theta=\frac {3-\gamma}2,$ if $\gamma\le 3$,  completes the proof.
\end{proof}

Thanks to Theorem \ref{tm-improve} and Lemma \ref{lm-improv-order}, we have the following overall estimate. It can be seen  that in the mean-squre sense, we only need three additional regularity on the initial data to derive the first order convergence w.r.t.\ $\tau$.
Note that in the existing results \cite{cui2017stochastic,belaouar2015numerical,cohen2017exponential}, one usually requires higher regularity assumption to prove first order convergence.

\begin{corollary}\label{cor-1}
  Let $T>0$, $\bs\in \mathbb N,\gamma \in \mathbb N^+,$   and $u_0\in L^2$.
 There exists a constant  $\epsilon>0$ such that if $\|u_0\|_{\dot H^{\bs+\gamma}}\le \epsilon,$
then for 
\eqref{wong-zakai} with $R\in [\sqrt{8|\ln(\delta)|},+\infty]$, $\delta\in(0,\tau]$ with $0<\tau=\frac {T}{K}$ and  $n\le K$, and $N\in \mathbb N^+\cup\{+\infty\}$ 
satisfying that $ 
 R\epsilon$ and $\tau$ are sufficient small,
it holds that  
    \begin{align*}
&\sup_{n\le K}\left\|u(t_n)-u^{\delta,R,N}_n\right\|_{L^2(\Omega;H^{\bs})} \lesssim  \tau^{\max(\frac {\gamma+1}4,1)}+\delta^{\max (\frac{\gamma}4,1)}+N^{-\gamma}. 
    \end{align*} 
\end{corollary}

\begin{rk}
    Another way to improve the order of the low-regularity integrator is replacing $I_{n,\tau}$
by a higher-order approximation, such as   
$$\int_0^{\tau}e^{-2\bi k k_1(B^{\delta,R}(t_n+s)-B^{\delta,R}(t_n))}+e^{2\bi k_2k_3(B^{\delta,R}(t_n+s)-B^{\delta,R}(t_n))}ds-\tau.$$
As a consequence, an order $1$ scheme (w.r.t. $\tau$) in stochastic case can be obtained via similar arguments by requiring only one additional bounded derivative. This kernel approximation also lends itself naturally to the construction of mass-preserving numerical methods, which will be reported in further detail in forthcoming work.
\end{rk}

\section{Numerical experiments}
\label{sec-4}
Let us now demonstrate the practical performance of SDLRI \eqref{eqn:1st_order_lri}. All experiments were conducted using a Matlab implementation of the corresponding algorithms, which will be made available at \cite{glimpse}. In the following experiments we work with low-regularity data of the following form: First we fix the number of Fourier modes, $2N+1$, in the spatial discretisation and take a sample of a vector of complex numbers sampled uniformly at random:
\begin{align*}
	\mathbf{U}=\left(U_{-N},\dots, U_{N}\right), \quad U_j\sim U([0,1+\bi ]),\,\, j=-N,\dots, N.
\end{align*}
Our initial condition is then given by its Fourier coefficients for a specified value of $\theta$:
\begin{align}\label{eqn:form_of_random_initial_data}
	\hat{u}^{0}_{-N+j}=\langle -N+j\rangle^{-\theta} U_{-N+j},\quad j=0,\dots, 2N,\, \text{where\ }\langle m\rangle=\begin{cases}|m|,&m\neq 0,\\
		1,& m=0.
	\end{cases}
\end{align}
In what follows we use two reference methods from prior work.

\begin{itemize}
    \item \textbf{The relaxed CN method} from \cite{Besse2004,belaouar2015numerical}, in the following form:
\begin{align*}\begin{cases}\frac{\phi^{n+\frac{1}{2}}+\phi^{n-\frac{1}{2}}}{2}=\left|u^n\right|^{2}, \\ u^{n+1}=u^n+\bi(B(t_{n+1})-B(t_n))\Delta\left(\frac{u^{n+1}+u^n}{2}\right)+\bi \tau\lambda\left(\frac{u^{n+1}+u^n}{2}\right) \phi^{n+\frac{1}{2}}.\end{cases}
\end{align*}
\item\textbf{A Wong--Zakai version of the exponential Euler method}:
\begin{align*}
    &u^{n+1}=e^{\bi \psi(\tau)\Delta}u_n\\
    &\ \ +\bi \tau \lambda e^{\bi \psi(\tau)\Delta}\left(\sum_{k=1}^{\lfloor \frac {\tau}{\delta} \rfloor} e^{-\bi \psi(\tau_{k-1})\Delta}(\tau_{k}-\tau_{k-1})\varphi_1(-\bi (\psi(\tau_{k})-\psi(\tau_{k-1}))\Delta)\right)|u_n|^2u_n,
\end{align*}
where we write $\psi(s)=B^{\delta,+\infty}(t_n+s)-B^{\delta,\infty}(t_n)$  and $0=\tau_0<\tau_1<\cdots<\tau_{\lfloor \frac {\tau}{\delta} \rfloor}=\tau$ is such that the Wong--Zakai approximation of $B$ is piecewise linear on $[\tau_{l},\tau_{l+1}]$ with $l\le \lfloor \frac {\tau}{\delta} \rfloor$.
\end{itemize}

In the following experiments the reference solution is computed with a split step method (cf. \cite{MR4278943}) with the timestep $\tau_{ref}=10^{-4}, N_{ref}=1024$. In  practice we did not need to impose a restriction on $R$ as the observations of the Brownian motion that we used in our experiments did not grow too rapidly and so we took $R=\infty$ for the below experiments. The initial data was rescaled to $\|u_0\|_{L^2}=0.1$, however the scaling was not relevant in practice and we found qualitatively similar results also for larger data.
\subsection{Validation on deterministic regime}
To begin with, we exhibit that our methods behave as expected in the ``classical'' setting when $\delta=T=1$. In this case, we approximate the Brownian motion with a linear function $B(0)+({B(1)-B(0)})s$ on the entire time interval. The result can be seen in Figure~\ref{fig:verification} and shows that our methods behave as expected - clear first order convergence for the exponential method, Lie Splitting and SDLRI and clear second order for the relaxed Crank--Nicolson method.
\begin{figure}[h!]
    \centering
\includegraphics[width=0.42\textwidth]{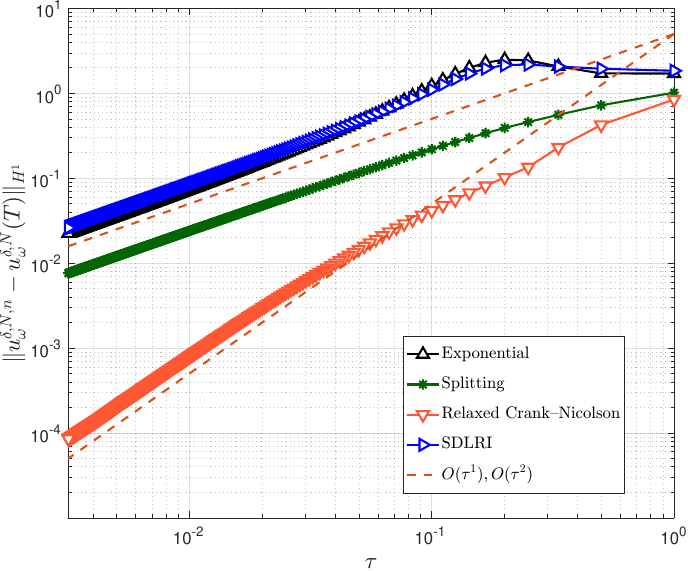}
    \caption{The approximation error under $H^1$-norm in the linear dispersion regime $\delta=1, R=\infty, N=512$ for a single trajectory with $C^{\infty}$-data.}
\label{fig:verification}
\end{figure}
\newpage\subsection{Strong convergence to Wong--Zakai approximation}
As the second numerical experiment, we study the strong error of the new integrator in light of Theorem~\ref{tm-improve}. For this we plot the quantity
\begin{align*}
&\left\|u^{\delta,R,N}(T)-u^{\delta,R,N}_{K}\right\|_{L^2(\Omega;H^{\bs})},
\end{align*}
computed using 60 random samples of the Brownian motion. The results for initial data of a range of different levels of regularity can be seen in Figure~\ref{fig:varying_regularity_512} which shows that indeed our methods converge as predicted in Theorem~\ref{tm-improve}. In fact, since we measure the error in $H^1$ by Theorem~\ref{tm-improve} we expect the method to converge at least at order $\mathcal{O}(\tau^{1/2})$ for data in $H^2$ and at order $\mathcal{O}(\tau)$ as soon as the data is at least in $H^4$, which can be observed in practice. On the other hand the reference methods clearly exhibit worse convergence behaviour in these low-regularity regimes.
\begin{figure}[h!]
    \centering
    \begin{subfigure}{0.495\textwidth}
\includegraphics[width=0.95\textwidth]{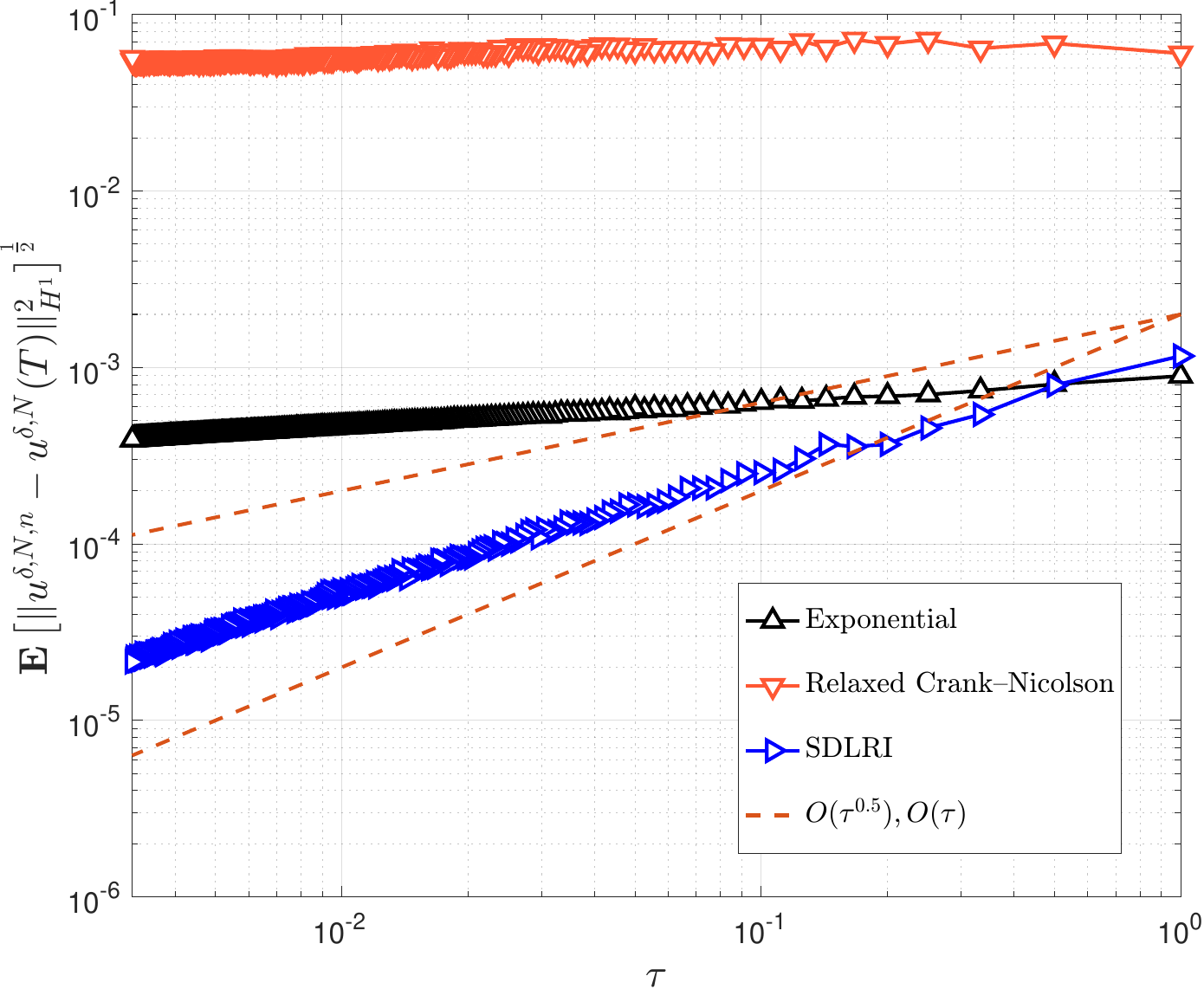}
		\caption{$H^2$ data.}
	\end{subfigure}%
    \begin{subfigure}{0.495\textwidth}
\includegraphics[width=0.95\textwidth]{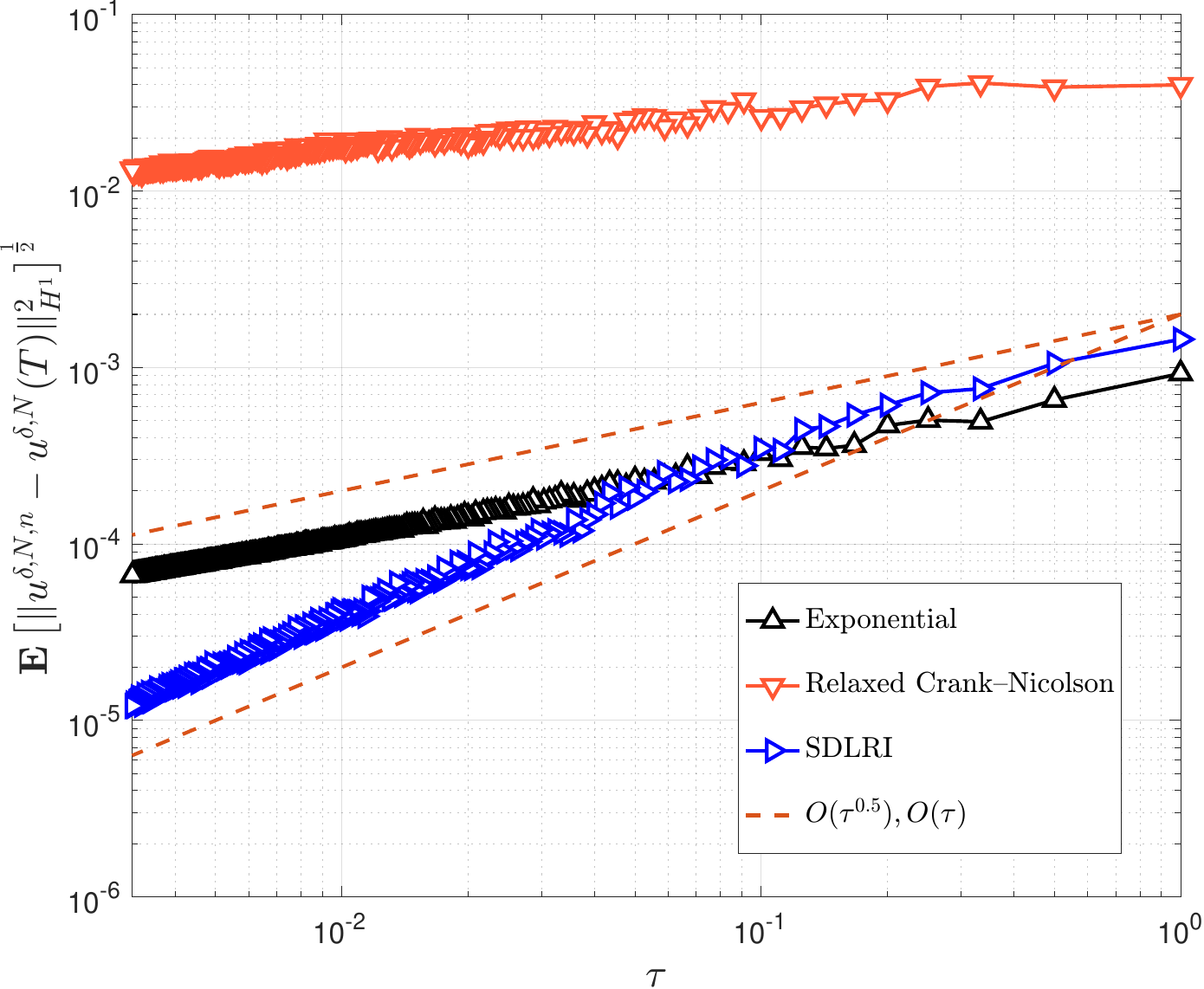}
		\caption{$H^3$ data.}
	\end{subfigure}\\
        \begin{subfigure}{0.495\textwidth}
\includegraphics[width=0.95\textwidth]{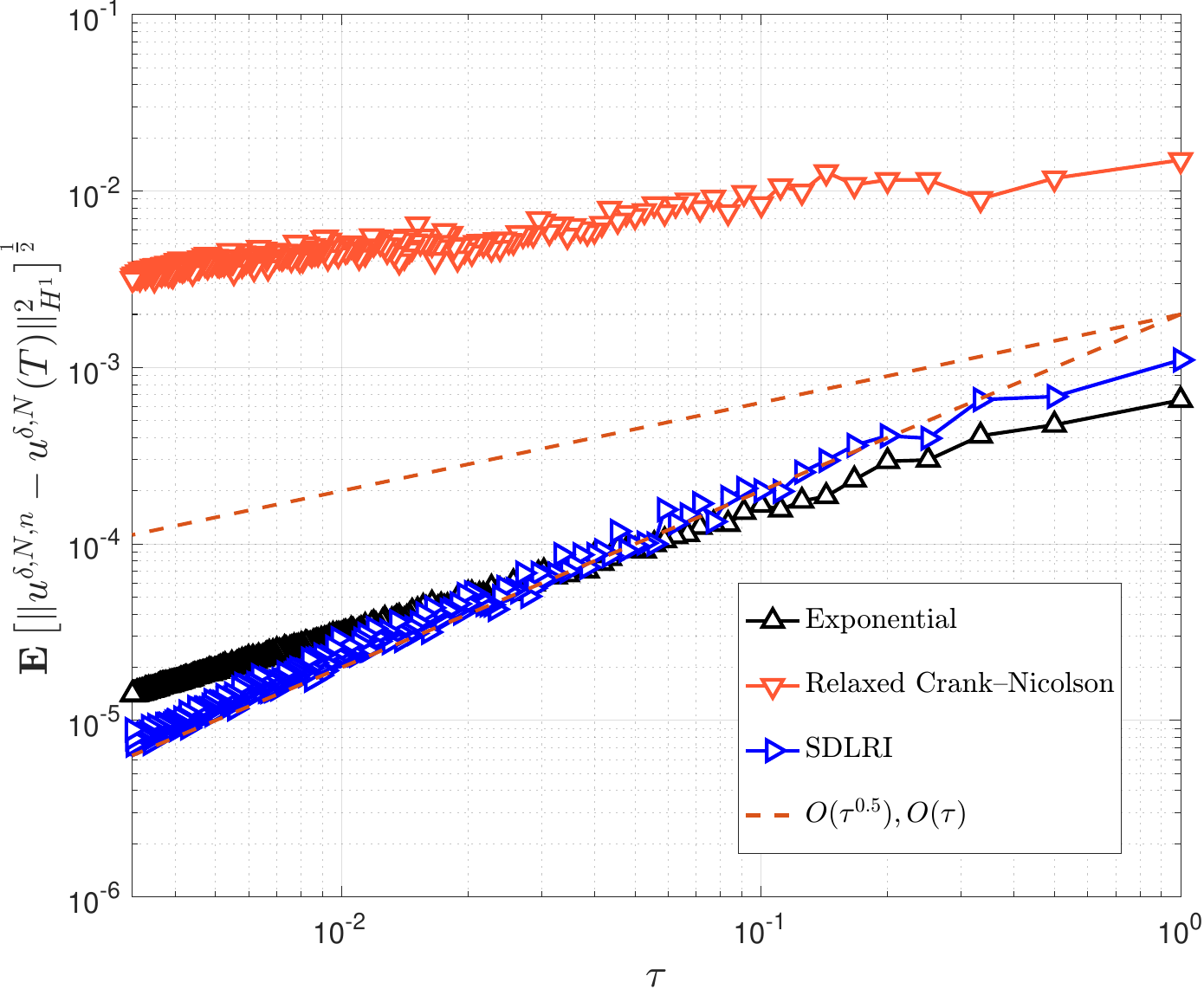}
		\caption{$H^4$ data.}
	\end{subfigure}%
        \begin{subfigure}{0.495\textwidth}
\includegraphics[width=0.95\textwidth]{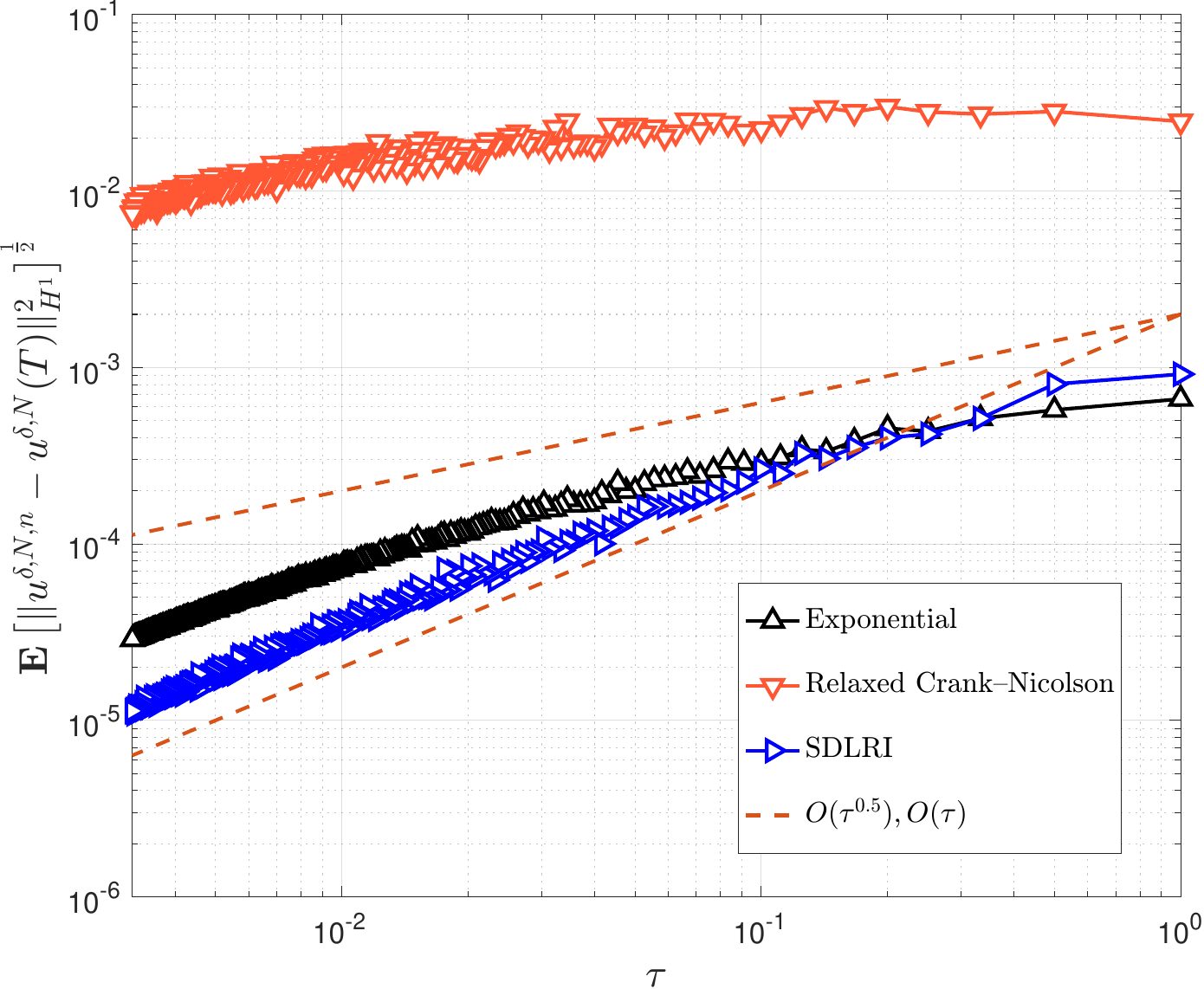}
		\caption{$C^{\infty}$ data.}
	\end{subfigure}%
    \caption{Strong error plots with $N=512$ and $\delta=2^{-12}$ for varying regularity of $u_0$ and $T=1$.}
    \label{fig:varying_regularity_512}
\end{figure}
\newpage\subsection{Pathwise convergence to Wong--Zakai approximation}
As a third experiment we validate the pathwise convergence result exhibited in Theorem~\ref{thm:global_error_pathwise_convergence_wong-zakai}. Indeed let us fix a path $\omega$ and a corresponding Wong--Zakai approximation $u^{\delta,N}$ with $\delta<1$ as specified below. In Figures~\ref{fig:varying_delta_512} \& \ref{fig:varying_regularity_512_pathwise} we observe the effect of the size of $\delta$ and the regularity of the solution on the convergence behaviour of our method SDLRI in the pathwise sense. In this experiment we include an additional reference method, the Lie splitting as studied in \cite{cui2017stochastic,MR4400428}:
\begin{align*}
    u^{n+1}=e^{\bi \psi(\tau)\Delta}e^{\bi \tau|u^n|^2}u^n,
\end{align*}
where we write $\psi(s)=B^{\delta,+\infty}(t_n+s)-B^{\delta,+\infty}(t_n)$. We observe that our method performs as expected but that, perhaps somewhat surprisingly, the splitting method appears to converge well also in the low-regularity regime when $\delta\ll 1$. This improved convergence rate will be studied in future work.

\begin{figure}[h!]
    \centering
    \begin{subfigure}{0.495\textwidth}
\includegraphics[width=0.85\textwidth]{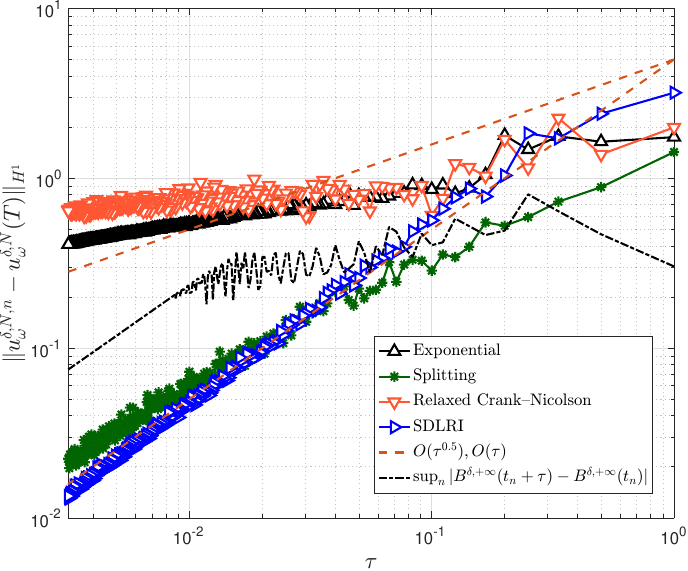}
		\caption{$\delta=2^{-6}$.}
	\end{subfigure}%
    \begin{subfigure}{0.495\textwidth}
\includegraphics[width=0.85\textwidth]{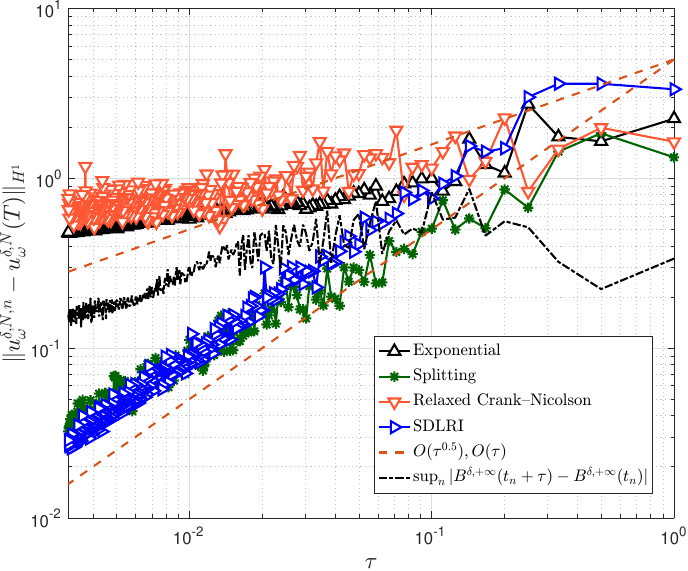}
		\caption{$\delta=2^{-9}$.}
	\end{subfigure}\\
        \begin{subfigure}{0.495\textwidth}
\includegraphics[width=0.85\textwidth]{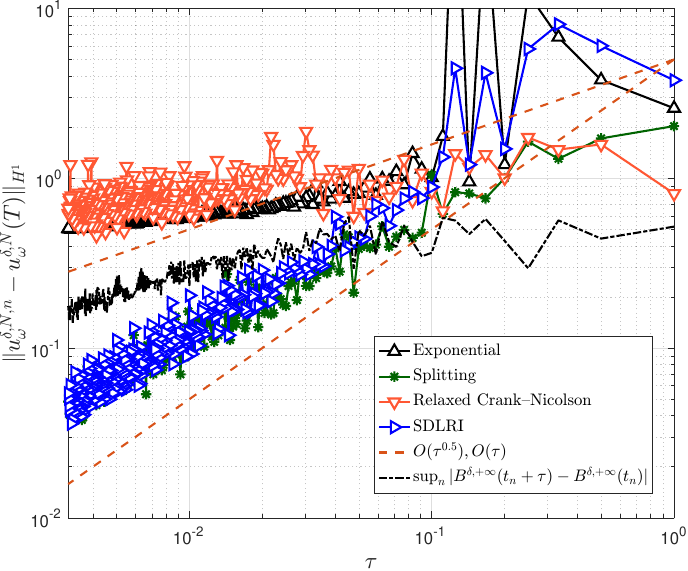}
		\caption{$\delta=2^{-12}$.}
	\end{subfigure}%
    \caption{Pathwise convergence plots with $N=512$ and $H^2$-data for several choices of Wong--Zakai parameter $\delta$.}
    \label{fig:varying_delta_512}
\end{figure}

\newpage
\begin{figure}[h!]
    \centering
    \begin{subfigure}{0.495\textwidth}
\includegraphics[width=0.85\textwidth]{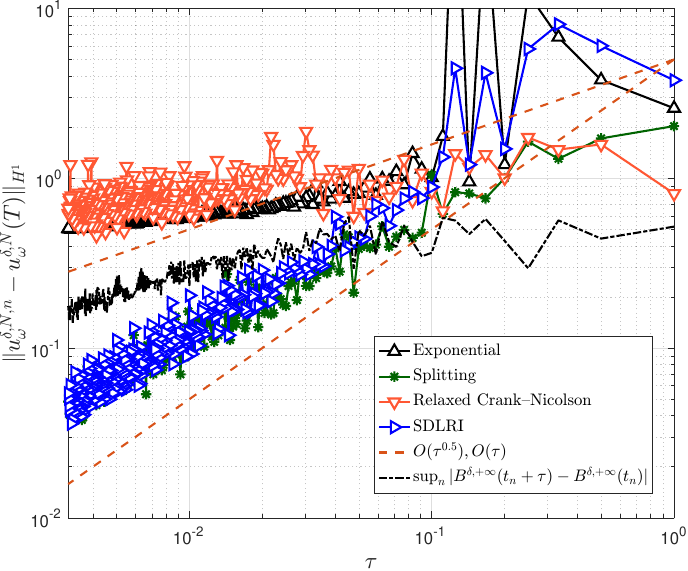}
		\caption{$H^2$ data.}
	\end{subfigure}%
    \begin{subfigure}{0.495\textwidth}
\includegraphics[width=0.85\textwidth]{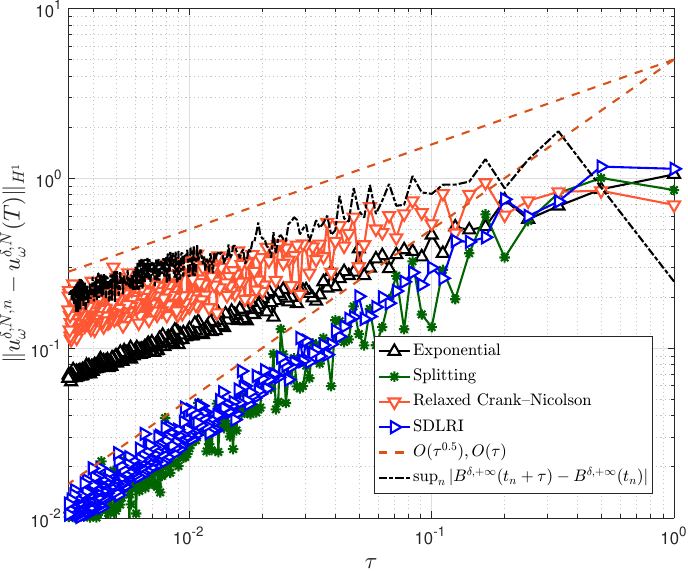}
		\caption{$H^3$ data.}
	\end{subfigure}\\
        \begin{subfigure}{0.495\textwidth}
\includegraphics[width=0.85\textwidth]{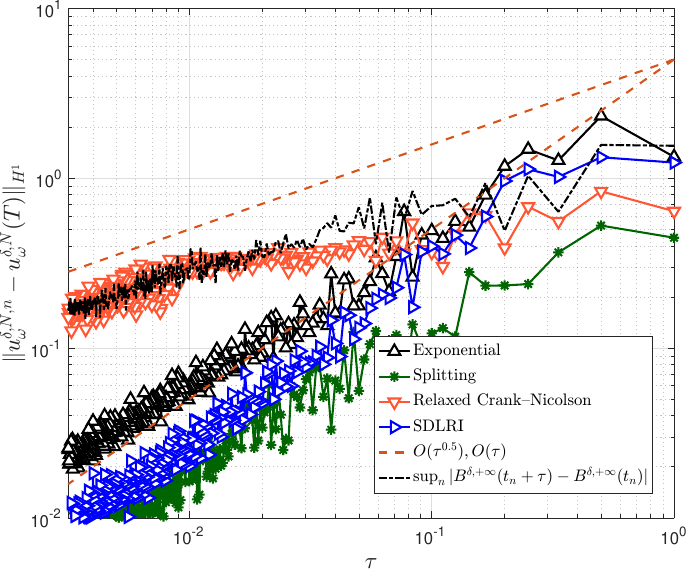}
		\caption{$C^{\infty}$ data.}
	\end{subfigure}%
    \caption{Pathwise convergence plots with $N=512$ and $\delta=2^{-12}$ for varying regularity of $u_0$.}
    \label{fig:varying_regularity_512_pathwise}
\end{figure}
\newpage\ \newpage\subsection{Convergence with respect to Wong--Zakai parameter}
As a final numerical experiment we examine the convergence of the numerical approximation to the original solution as a function of the Wong--Zakai parameter $\delta$ as estimated in Proposition~\ref{prop:approx_prop_wong_zakai_periodic}. To understand this behavior in further detail it is helpful to consider the approximation rate estimated in Lemma \ref{hold-B}.

The results of the numerical experiment are shown in Figure~\ref{fig:varying_regularity_and_delta_512} and are found to be in good agreement with Proposition~\ref{prop:approx_prop_wong_zakai_periodic}.

\begin{figure}[h!]
    \centering
    \begin{subfigure}{0.495\textwidth}
\includegraphics[width=0.85\textwidth]{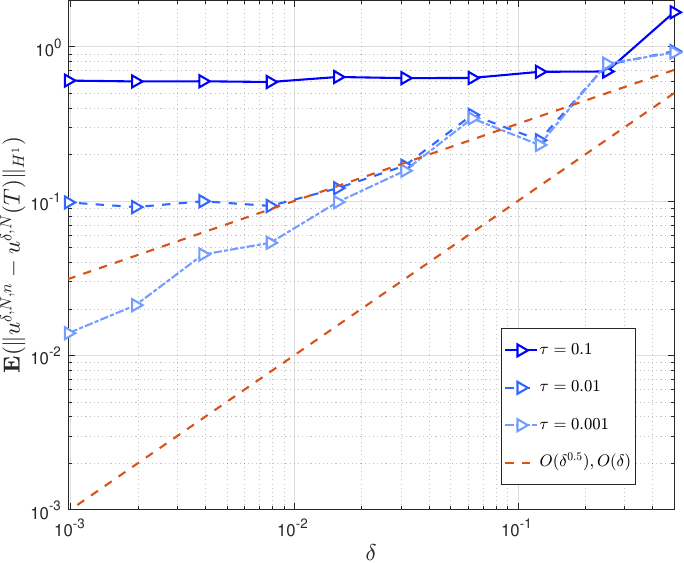}
		\caption{$H^2$ data.}
	\end{subfigure}%
    \begin{subfigure}{0.495\textwidth}
\includegraphics[width=0.85\textwidth]{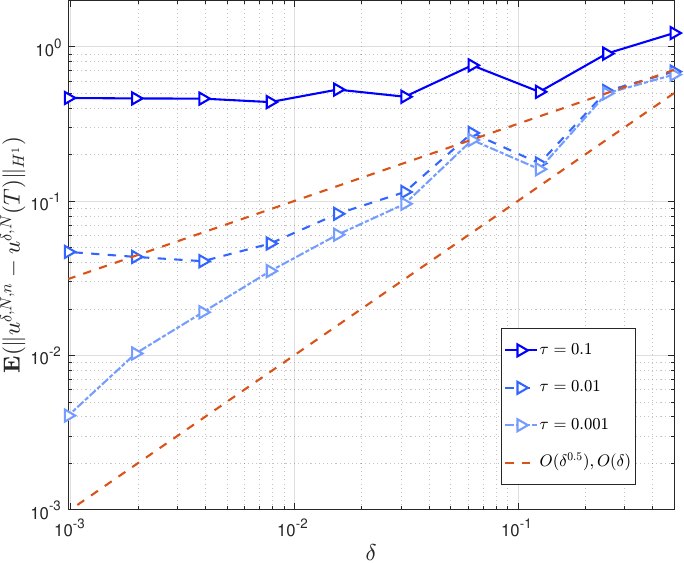}
		\caption{$H^3$ data.}
	\end{subfigure}\\
        \begin{subfigure}{0.495\textwidth}
\includegraphics[width=0.85\textwidth]{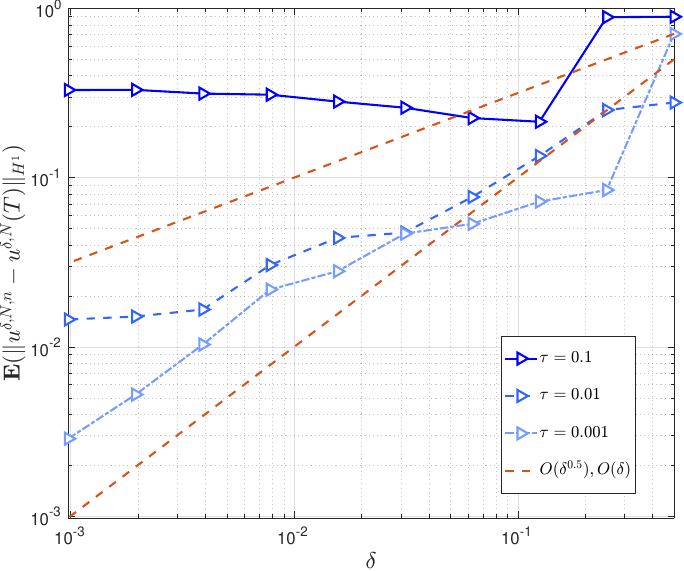}
		\caption{$C^{\infty}$ data.}
	\end{subfigure}%
    \caption{Error in expectation for SDLRI \eqref{eqn:1st_order_lri} a function of $\delta$ for $N=512$ and $\tau=0.1,0.01,0.001$ for varying regularity of $u_0$.}
    \label{fig:varying_regularity_and_delta_512}
\end{figure}

\newpage\section{Conclusions and future work}\label{sec:conclusions} In this paper, we introduced the first low-regularity integrator for the NLSE with white noise dispersion. Our construction is based on the Wong--Zakai approximation which leads to an amenable form of the nonlinear frequency interactions with linear phase in Duhamel's formula. This facilitates the design of a resonance-based method which leads to favorable low-regularity convergence properties. We then further exploited this special form of the numerical method to derive improved strong convergence estimates for this new integrator which are, to the best of our knowledge, the first strong error estimates for a method for \eqref{model2}. These results are supported in numerical experiments that also highlight the favorable practical low-regularity convergence properties of our new method.
Based on promising recent work in \cite{arXiv:2410.22359} and \cite{maierhofer2023bridging}, our future work will focus on the development of structure-preserving algorithms for \eqref{model2}.

\section*{Acknowledgments}
{The authors would like to thank Adrien Busnot Laurent (INRIA Rennes) for several helpful discussions. {Part of this work was undertaken while the authors were participating in the Workshop on Structure-preserving numerical methods and their applications at the Hong Kong Polytechnic University, in July 2024.} GM gratefully acknowledges funding from the Mathematical Institute, University of Oxford. The research is partially supported by research grants NSF DMS-2307465 and ONR N00014-21-1-2891. The research of the first author is partially supported by the Hong Kong Research Grant Council GRF grant 15302823, NSFC grant 12301526, NSFC/RGC Joint Research Scheme N$\_$PolyU5141/24, MOST National Key $R\&D$ Program No. 2024YFA1015900, the internal grants 
(P0045336, P0046811) from Hong Kong Polytechnic University and the CAS AMSS-PolyU Joint Laboratory of Applied Mathematics.}


\bibliographystyle{amsplain}
\bibliography{references}

\providecommand{\bysame}{\leavevmode\hbox to3em{\hrulefill}\thinspace}
\providecommand{\MR}{\relax\ifhmode\unskip\space\fi MR }
\providecommand{\MRhref}[2]{%
  \href{http://www.ams.org/mathscinet-getitem?mr=#1}{#2}
}
\providecommand{\href}[2]{#2}
\begin{thebibliography}{10}

\bibitem{Agr21}
G.~P. Agrawal, \emph{Applications of nonlinear fiber optics}, Academic Press,
  San Diego, 2001.

\bibitem{alama2023symmetric}
Y.~Alama~Bronsard, \emph{{A symmetric low-regularity integrator for the
  nonlinear Schr{\"o}dinger equation}}, IMA J. Numer. Anal. (2023), drad093.

\bibitem{bronsard2023error}
\bysame, \emph{{Error analysis of a class of semi-discrete schemes for solving
  the Gross--Pitaevskii equation at low regularity}}, J. Comput. Appl. Math.
  \textbf{418} (2023), 114632.

\bibitem{alamabronsard_et_al23}
Y.~Alama~Bronsard, Y.~Bruned, G.~Maierhofer, and K.~Schratz, \emph{{Symmetric
  resonance based integrators and forest formulae}}, arXiv preprint
  arXiv:2305.16737 (2023).

\bibitem{banicamaierhoferschratz22}
V.~Banica, G.~Maierhofer, and K.~Schratz, \emph{{Numerical integration of
  Schr{\"o}dinger maps via the Hasimoto transform}}, SIAM J. Numer. Anal.
  \textbf{62} (2024), no.~1, 322--352.

\bibitem{belaouar2015numerical}
R.~Belaouar, A.~de~Bouard, and A.~Debussche, \emph{{Numerical analysis of the
  nonlinear Schr{\"o}dinger equation with white noise dispersion}}, Stochastic
  Partial Differential Equations: Analysis and Computations \textbf{3} (2015),
  no.~1, 103--132.

\bibitem{Besse2004}
C.~Besse, \emph{{A Relaxation Scheme for the Nonlinear Schr\"odinger
  Equation}}, SIAM Journal on Numerical Analysis \textbf{42} (2004), no.~3,
  934--952.

\bibitem{MR1616917}
J.~Bourgain, \emph{Refinements of {S}trichartz' inequality and applications to
  {$2$}{D}-{NLS} with critical nonlinearity}, Internat. Math. Res. Notices
  (1998), no.~5, 253--283. \MR{1616917}

\bibitem{MR4400428}
C.-E. Br\'{e}hier and D.~Cohen, \emph{Strong rates of convergence of a
  splitting scheme for {S}chr\"{o}dinger equations with nonlocal interaction
  cubic nonlinearity and white noise dispersion}, SIAM/ASA J. Uncertain.
  Quantif. \textbf{10} (2022), no.~1, 453--480. \MR{4400428}

\bibitem{bruned_schratz_2022}
Y.~Bruned and K.~Schratz, \emph{Resonance-based schemes for dispersive
  equations via decorated trees}, Forum Math. Pi \textbf{10} (2022), e2 1--76.

\bibitem{MR4250287}
I.~C\^{i}mpean and A.~Grecu, \emph{The nonlinear {S}chr\"{o}dinger equation
  with white noise dispersion on quantum graphs}, Commun. Math. Sci.
  \textbf{19} (2021), no.~2, 405--435. \MR{4250287}

\bibitem{cohen2017exponential}
D.~Cohen and G.~Dujardin, \emph{Exponential integrators for nonlinear
  schr{\"o}dinger equations with white noise dispersion}, Stochastics and
  Partial Differential Equations: Analysis and Computations \textbf{5} (2017),
  no.~4, 592--613.

\bibitem{MR3826675}
J.~Cui and J.~Hong, \emph{Analysis of a splitting scheme for damped stochastic
  nonlinear {S}chr\"{o}dinger equation with multiplicative noise}, SIAM J.
  Numer. Anal. \textbf{56} (2018), no.~4, 2045--2069. \MR{3826675}

\bibitem{cui2017stochastic}
J.~Cui, J.~Hong, Z.~Liu, and W.~Zhou, \emph{{Stochastic symplectic and
  multi-symplectic methods for nonlinear Schr{\"o}dinger equation with white
  noise dispersion}}, Journal of Computational Physics \textbf{342} (2017),
  267--285.

\bibitem{MR4755536}
J.~Cui and L.~Sun, \emph{Quantifying the effect of random dispersion for
  logarithmic {S}chr\"{o}dinger equation}, SIAM/ASA J. Uncertain. Quantif.
  \textbf{12} (2024), no.~2, 579--613. \MR{4755536}

\bibitem{MR2652190}
A.~de~Bouard and A.~Debussche, \emph{The nonlinear {S}chr\"{o}dinger equation
  with white noise dispersion}, J. Funct. Anal. \textbf{259} (2010), no.~5,
  1300--1321. \MR{2652190}

\bibitem{MR2832639}
A.~Debussche and Y.~Tsutsumi, \emph{1{D} quintic nonlinear {S}chr\"{o}dinger
  equation with white noise dispersion}, J. Math. Pures Appl. (9) \textbf{96}
  (2011), no.~4, 363--376. \MR{2832639}

\bibitem{MR4288117}
S.~Dumont, O.~Goubet, and Y.~Mammeri, \emph{Decay of solutions to one
  dimensional nonlinear {S}chr\"{o}dinger equations with white noise
  dispersion}, Discrete Contin. Dyn. Syst. Ser. S \textbf{14} (2021), no.~8,
  2877--2891. \MR{4288117}

\bibitem{feng_maierhofer_schratz_2023}
Y.~Feng, G.~Maierhofer, and K.~Schratz, \emph{Long-time error bounds of
  low-regularity integrators for nonlinear {S}chr\"odinger equations}, Math.
  Comp. \textbf{93} (2023), no.~348, 1569--1598.

\bibitem{feng2024explicit}
Y.~Feng, G.~Maierhofer, and C.~Wang, \emph{{Explicit symmetric low-regularity
  integrators for the nonlinear Schr\"odinger equation}}, arXiv preprint
  arXiv:2411.07720 (2024).

\bibitem{arXiv:2410.22201}
S.~D. Giovacchino and K.~Schratz, \emph{{Long-term error analysis of
  low-regularity integrators for stochastic Schr\"odinger equations}},  (2024).

\bibitem{arXiv:2410.22359}
J.~A. Goodall and Y.~Bruned, \emph{{Low regularity symplectic schemes for
  stochastic NLS}}, arXiv:2410.22359.

\bibitem{arXiv:2312.16690}
\bysame, \emph{{Resonance based schemes for SPDEs}}, arXiv:2312.16690.

\bibitem{MR4167043}
M.~Hofmanov\'{a}, M.~Kn\"{o}ller, and K.~Schratz, \emph{Randomized exponential
  integrators for modulated nonlinear {S}chr\"{o}dinger equations}, IMA J.
  Numer. Anal. \textbf{40} (2020), no.~4, 2143--2162. \MR{4167043}

\bibitem{iserles2024accelerated}
A.~Iserles and G.~Maierhofer, \emph{{An accelerated Levin--Clenshaw--Curtis
  method for the evaluation of highly oscillatory integrals}},
  arXiv:2404.11448.

\bibitem{MR4048623}
A.~Laurent and G.~Vilmart, \emph{Multirevolution integrators for differential
  equations with fast stochastic oscillations}, SIAM J. Sci. Comput.
  \textbf{42} (2020), no.~1, A115--A139. \MR{4048623}

\bibitem{levin1982procedures}
D.~Levin, \emph{Procedures for computing one-and two-dimensional integrals of
  functions with rapid irregular oscillations}, Mathematics of Computation
  \textbf{38} (1982), no.~158, 531--538.

\bibitem{MR4312402}
B.~Li and Y.~Wu, \emph{A fully discrete low-regularity integrator for the 1{D}
  periodic cubic nonlinear {S}chr\"{o}dinger equation}, Numer. Math.
  \textbf{149} (2021), no.~1, 151--183. \MR{4312402}

\bibitem{glimpse}
G.~Maierhofer, \emph{{GLIMPSE}}, \url{https://github.com/GeorgAUT/GLIMPSE},
  2025.

\bibitem{maierhofer2023bridging}
G.~Maierhofer and K.~Schratz, \emph{{Bridging the gap: symplecticity and low
  regularity in Runge--Kutta resonance-based schemes}}.

\bibitem{MR4278943}
R.~Marty, \emph{Local error of a splitting scheme for a nonlinear
  {S}chr\"{o}dinger-type equation with random dispersion}, Commun. Math. Sci.
  \textbf{19} (2021), no.~4, 1051--1069. \MR{4278943}

\bibitem{ostermann2018low}
A.~Ostermann and K.~Schratz, \emph{{Low regularity exponential-type integrators
  for semilinear Schr{\"o}dinger equations}}, Foundations of Computational
  Mathematics \textbf{18} (2018), 731--755.

\bibitem{rousset2021general}
Fr{\'e}d{\'e}ric Rousset and Katharina Schratz, \emph{A general framework of
  low regularity integrators}, SIAM Journal on Numerical Analysis \textbf{59}
  (2021), no.~3, 1735--1768.

\bibitem{SL2017JDE}
J.~Shen and K.~Lu, \emph{Wong-{Z}akai approximations and center manifolds of
  stochastic differential equations}, J. Differential Equations \textbf{263}
  (2017), no.~8, 4929--4977. \MR{3680943}

\bibitem{Ste24}
G.~Stewart, \emph{On the wellposedness of periodic nonlinear {S}chr\"{o}dinger
  equations with white noise dispersion}, Stoch. Partial Differ. Equ. Anal.
  Comput. \textbf{12} (2024), no.~3, 1417--1438. \MR{4781788}

\bibitem{MR2233925}
T.~Tao, \emph{Nonlinear dispersive equations}, CBMS Regional Conference Series
  in Mathematics, vol. 106, Published for the Conference Board of the
  Mathematical Sciences, Washington, DC; by the American Mathematical Society,
  Providence, RI, 2006, Local and global analysis. \MR{2233925}

\bibitem{MR183023}
E.~Wong and M.~Zakai, \emph{On the relation between ordinary and stochastic
  differential equations}, Internat. J. Engrg. Sci. \textbf{3} (1965),
  213--229. \MR{183023}

\end{thebibliography}

\begin{appendix}
\section{Definition of the solution}
\label{sec-appendix-1}

We introduce the definition of mild solutions of the stochastic nonlinear Schr\"odinger equation with white noise dispersion (1.2) (see also \cite{Ste24}) and its Wong--Zakai approximation (2.5).

\begin{df}\label{def-1}
An $L^2$-valued $\mathcal F_t$-adapted process $u(t),t\in [0,T],$ is called a mild solution of (1.2) if 
\begin{align*}
 u(t)=e^{\bi \Delta B(t)}u_0+\bi \lambda \int_0^t e^{\bi \Delta(B(t)-B(s))} |u(s)|^2 u(s)ds, \text{a.s.}   
\end{align*}
for each $t\in[0,T].$ In particular, the appearing integrals have to be well-defined.
\end{df}

\begin{df}\label{def-2}
Let $\delta\in (0,1), R\in [1,+\infty], N\in\mathbb N^+\cup\{+\infty\}.$
An $L^2$-valued stochastic process $u^{\delta,R,N}(t),t\in [0,T],$ is called a mild solution of (2.5) if 
\begin{align*}
 u^{\delta,R,N}(t)=e^{\bi \Delta B^{\delta,R}(t)}u_0+\bi \lambda \int_0^t e^{\bi \Delta(B^{\delta,R}(t)-B^{\delta,R}(s))}  |u^{\delta,R,N}(s)|^2 u^{\delta,R,N}(s)ds, \text{a.s.}   
\end{align*}
for each $t\in[0,T].$ In particular, the appearing integrals have to be well-defined.
\end{df}

Note that in Definition \ref{def-1}, the solution $u(t)$ is $\mathcal F_t$-adapted if $u_0$ is $\mathcal F_0$-adaptive. In contrast, the Wong-Zakai approximation   $u^{\delta,R,N}(t)$ is $\mathcal F_{\lfloor \frac {t}\delta \rfloor\delta+\delta}$-adaptive  for any $t\in [0,T]$ due to the definition of $B^{\delta,R}$. In particular, on the time $i\delta, i\in \mathbb N, i\le \lfloor \frac {T}\delta \rfloor$, $u(i\delta)$ is $\mathcal F_{i\delta}$-adaptive.

\section{Well-posedness without the smallness condition}
\label{sec-appendix-2}

\subsection{Well-posedness for the Wong--Zakai approximation}

Following the idea in \cite{Ste24}, we show the global well-posedness of the Wong--Zakai approximation (2.5) with $R=+\infty$ 
in $L^4([0,T];L^4)\cap \mathcal C([0,T];L^2)$. 
Namely, we consider $u^{\delta,N}:=u^{\delta,+\infty,N}$ satisfying \eqref{wong-zakai}, i.e.
\begin{align*}
d u^{\delta,N}(t)=\bi \Delta u^{\delta,N}(t)\circ d B^{\delta}(t)+\bi \lambda \pi_N|u^{\delta,N}(t)|^2u^{\delta,N}(t) dt.
\end{align*}
where we denote $B^{\delta}:=B^{\delta,+\infty}$ for simplicity as the standard linear Wong--Zakai approximation of Brownian motion.

Notice that for $\psi\in L^2,$
\begin{align*}
e^{\bi B^{\delta}(t)\Delta}\psi (x)=\sum_{k\in \mathbb Z}a_ke^{\bi k x- \bi k^2  \frac {B(\delta)-B(0)}{\delta} t}, \;\forall \; t\in [0,\delta],
\end{align*} where $a_k=\frac 1{2\pi}\int_{\mathbb T} \psi(x) e^{-\bi kx} dx.$

\begin{lm}\label{sto-group}
For each interval
 $[i\delta,(i+1)\delta]$ with $i\in \mathbb N$ and any $\psi\in L^2$,  it holds that 
\begin{align}\label{l4tol2}
 \|e^{\bi[B^{\delta}(t)-B^{\delta}(i\delta)]} \psi \|_{L^4([i\delta,(i+1)\delta];L^4)}\lesssim \Big(\delta+\frac {\delta}{|B((i+1)\delta)-B(i\delta)|}\Big)^{\frac 14}\|\psi\|_{L^2}, \; \text{a.s.}  
\end{align}
\end{lm}

\begin{proof}
Due to the definition of $B^{\delta},$ it suffices to prove the case that $i=0.$
 Using the fact that $\|f\|_{L^4}=\|f \bar f\|_{L^2}^{\frac 12}$ for any $f\in L^4$ and the Parseval's inequality, 
it holds that 
\begin{align*}
&\int_0^{\delta}\int_{\mathbb T}\Big(\sum_{k}|a_k|^2+\sum_{k_1\neq k_2}a_{k_1}\bar a_{k_2}e^{\bi(k_1-k_2)x-\bi(k_1-k_2)(k_1+k_2) \frac {B(\delta)}{\delta}t}\Big)^{2} dx dt\\
   &\le \delta \|\psi\|_{L^2}^4
   +2\|\psi\|_{L^2}^2 \int_0^{\delta}\int_{\mathbb T}\sum_{k_1\neq k_2} a_{k_1}\bar a_{k_2}e^{\bi(k_1-k_2)x-\bi(k_1-k_2)(k_1+k_2) \frac {B(\delta)}{\delta}t}dxdt\\
   &+\int_0^{\delta}\int_{\mathbb T}
   \sum_{k_1\neq k_2, k_3\neq k_4}a_{k_1}\bar a_{k_2}a_{k_3}\bar a_{k_4} e^{\bi(k_1-k_2)x+\bi(k_3-k_4)x-\bi(k_1-k_2)(k_1+k_2) \frac {B(\delta)}{\delta}t-\bi(k_3-k_4)(k_3+k_4) \frac {B(\delta)}{\delta}t}dxdt.
   \end{align*}
H\"older's inequality and the fact that 
 $|e^{\bi y}|=1 $ for any $y\in\mathbb R,$ yield that 
   \begin{align*}
   &\|e^{\bi B^{\delta}(t)\Delta}\psi\|_{L^4([0,\delta],L^4)}^4\lesssim  \delta \|\psi\|_{L^2}^4+\Big|\int_0^{\delta}
   \sum_{k_1,nl=0} a_{k_1}\bar a_{k_1+l}a_{k_1+n+l}\bar a_{k_1+n}dt\Big|\\
   &+ \Big|\int_0^{\delta}
   \sum_{k_1,l\neq 0,n\neq 0}a_{k_1}\bar a_{k_1+l}a_{k_1+n+l}\bar a_{k_1+n}e^{-\bi 2 ln \frac {B(\delta)}{\delta}t}dt\Big|\\
   &\lesssim  \delta \|\psi\|_{L^2}^4+\frac {\delta}{|B(\delta)|} \sum_{k_1,l\neq 0,n\neq 0} \frac 1{ln}a_{k_1}\bar a_{k_1+l}a_{k_1+n+l}\bar a_{k_1+n}\\
   &\lesssim  \delta \|\psi\|_{L^2}^4
   +\frac {\delta}{|B(\delta)|} 
   \sqrt{\sum_{k_1} |a_{k_1}|^2}
   \sqrt{\sum_{k_1} \Big(\sum_{l\neq0,n\neq 0} \frac 1{nl} \bar a_{k_1+l}\bar a_{k_1+n}a_{k_1+n+l}\Big)^2} \\
   &\lesssim \delta \|\psi\|_{L^2}^4
   +\frac {\delta}{|B(\delta)|} \|\psi\|_{L^2}   \sqrt{\sum_{k_1}
  \sum_{l\neq 0,n\neq 0} \frac 1{l^2 n^2} |a_{k_1+n+l}|^2
   \sum_{l\neq 0} |a_{k_1+l}|^2\sum_{n\neq 0}| a_{k_1+n}|^2}\\
   &\lesssim \delta \|\psi\|_{L^2}^4
   +\frac {\delta}{|B(\delta)|} \|\psi\|_{L^2}^4. 
\end{align*}
\end{proof}
The second ingredient for the well-posedness is to deal with the inhomogeneous part of \eqref{wong-zakai}. 
We also illustrate this estimate on the initial interval $[0,\delta].$

Let $f\in L^{\frac 43}([0,\delta],L^{\frac 43})$ and denote $f(s,x)=\sum_{k}a_k(s)e^{\bi k x}.$
It can be seen that 
\begin{align*}
&\Big\|\int_0^t e^{\bi(B^{\delta}(t)-B^{\delta}(s))\Delta} f(s)ds\Big\|_{L^{4}([0,\delta],L^{4})}^4
=\int_0^{\delta}
\int_{[0,t]^4}\sum_{k,n,l}A_{k,n,l}(\vec{s})E_{k,n,l}(t, \vec{s})
d\vec{s}dt\\
&=\int_0^{\delta}
\int_{[0,t]^4}\sum_{k,n\neq 0, l\neq 0}A_{k,n,l}(\vec{s})E_{k,n,l}(t, \vec{s})
d\vec{s}dt
+\int_0^{\delta}
\int_{[0,t]^4}\sum_{k,nl=0}A_{k,n,l}(\vec{s})E_{k,n,l}(t, \vec{s})
d\vec{s}dt\\
&-\int_0^{\delta}
\int_{[0,t]^4}\sum_{k}A_{k,0,0}(\vec{s})E_{k,0,0}(t, \vec{s})
d\vec{s}dt,
\end{align*}
where $\vec{s}=(s_1,s_2,s_3,s_4),$
\begin{align*}
A_{k,n,l}(\vec{s})&=a_{k}(s_1)\bar a_{k-l}(s_2)a_{k-l-n}(s_3)\bar a_{k-n}(s_4),\\
E_{k,n,l}(t,\vec{s})&=e^{\bi \frac {B(\delta)}{\delta}\Big((t-s_1)k^2-(t-s_2)(k-l)^2-(t-s_3)(k-l-n)^2+(t-s_4)(k-n)^2\Big)}.
\end{align*}
Using the Hausdorff--Young inequality and H\"older's inequality, it follows that 
\begin{align*}
&\int_0^{\delta}
\int_{[0,t]^4}\sum_{k}A_{k,0,0}(\vec{s})E_{k,0,0}(t, \vec{s})
d\vec{s}dt\\
&\le C \int_0^{\delta}\int_{[0,t]^4} \|a_{k}(s_1)\|_{l^4_k}\|a_{k}(s_1)\|_{l^4_k}\|a_{k}(s_1)\|_{l^4_k}\|a_{k}(s_1)\|_{l^4_k}d\vec{s}dt\\
&\le C  \int_0^{\delta}\int_{[0,t]^4} \|f(s_1)\|_{L^{\frac 43}}\|f(s_2)\|_{L^{\frac 43}} \|f(s_3)\|_{L^{\frac 43}} \|f(s_4)\|_{L^{\frac 43}} d\vec{s} 
dt\\
&\le C \delta^2 
(\int_0^{\delta} \|f(t)\|_{L^{\frac 43}}^{\frac 43}dt)^{3}
\end{align*}
Using again H\"older's inequality and Parseval's inequality, 
we have that 
\begin{align*}
\int_0^{\delta}
&\int_{[0,t]^4}\sum_{k,nl=0}A_{k,n,l}(\vec{s})E_{k,n,l}(t, \vec{s})
d\vec{s}dt\\
&= 2\int_0^{\delta}\Big|\int_{[0,t]^2}\sum _{k}a_k(s_1)\bar a_k(s_2)e^{\bi \frac {B(\delta)}{\delta}(s_2-s_1)k^2} ds_1ds_2\Big|^2dt\\
&\le C \delta  \Big(\int_0^{\delta} \|f(t)\|_{L^{\frac 43}}^{\frac 34}dt\Big)^{\frac 32}\Big\|\int_0^t e^{\bi(B^{\delta}(t)-B^{\delta}(s))\Delta} f(s)ds\Big\|_{L^{4}([0,\delta],L^{4})}^2\\
&\le C \delta  \Big(\int_0^{\delta} \|f(t)\|_{L^{\frac 43}}^{\frac 34}dt\Big)^{3}+\frac 12\Big\|\int_0^t e^{\bi(B^{\delta}(t)-B^{\delta}(s))\Delta} f(s)ds\Big\|_{L^{4}([0,\delta],L^{4})}^4.
\end{align*}

It suffices to estimate the term $\int_0^{\delta}
\int_{[0,t]^4}\sum\limits_{k,n\neq 0, l\neq 0}A_{k,n,l}(\vec{s})E_{k,n,l}(t, \vec{s})
d\vec{s}dt.$
Using the integration by parts formula,  we get 
\begin{align*}
&\int_0^{\delta}
\int_{[0,t]^4}\sum_{k,n\neq 0, l\neq 0}A_{k,n,l}(\vec{s})E_{k,n,l}(t, \vec{s})
d\vec{s}dt\\
&= 24 \sum_{k,n\neq 0, l\neq 0} \int_0^{\delta}
E_{k,n,l}(t,\vec{0}) 
\int_{0}^t\int_0^{s_1}\int_{0}^{s_2}\int_0^{s_3}
 A_{k,n,l}(\vec{s})  E_{k,n,l}(0,\vec{s}) d\vec{s}dt\\
&=- 24 \sum_{k,n\neq 0, l\neq 0} \frac {1}{\bi2 nl}\frac {\delta}{B(\delta)} e^{-2\bi nl B(\delta)}
\int_0^{\delta}\int_0^{s_1}\int_{0}^{s_2}\int_0^{s_3}A_{k,n,l}(\vec{s})  E_{k,n,l}(0,\vec{s}) d\vec{s}\\
&+24\int_0^{\delta}\sum_{k,n\neq 0, l\neq 0} \frac {1}{\bi 2 nl}\frac {\delta}{B(\delta)} e^{-\bi 2nl \frac {B(\delta)}{\delta} t}\int_0^{t}\int_{0}^{s_2}\int_0^{s_3}a_{k}(t)\bar a_{k-l}(s_2)a_{k-l-n}(s_3)\\
&\qquad \bar a_{k-n}(s_4) e^{\bi \frac {B(\delta)}{\delta}\Big(-t k^2+s_2(k-l)^2+s_3(k-l-n)^2-s_4(k-n)^2\Big)} ds_4ds_3ds_2dt.
\end{align*}
According to the fact that $|e^{\bi y}|=1 $ for any $y\in\mathbb R,$  using the Hausdorff--Young inequality,  we further obtain
\begin{align*}
&\Big|\int_0^{\delta}
\int_{[0,t]^4}\sum_{k,n\neq 0, l\neq 0}A_{k,n,l}(\vec{s})E_{k,n,l}(t, \vec{s})
d\vec{s}dt\Big|\\
&
\le 
C \frac {\delta}{|B(\delta)|}
\sum_{k,n\neq 0, l\neq 0}\frac {1}{nl}
\int_0^{\delta}\int_0^{s_1}\int_{0}^{s_2}\int_0^{s_3}|A_{k,n,l}(\vec{s})|  d\vec{s}
\\
&+C \frac {\delta}{|B(\delta)|}\sum_{k,n\neq 0, l\neq 0}\frac {1}{nl} \int_0^{\delta}\int_0^{t}\int_{0}^{s_2}\int_0^{s_3}|a_{k}(t)| |a_{k-l}(s_2)||a_{k-l-n}(s_3)||a_{k-n}(s_4)|ds_4ds_3ds_2dt\\
&\le C \frac {\delta}{|B(\delta)|}\sum_{ n\neq 0, l\neq 0} \frac 1{nl} (\int_0^\delta \|f(t)\|_{L^{\frac 43}}^{\frac 43}dt)^{3}.
\end{align*}

Combining the above estimates together and truncating the Fourier series by the projection $\pi_N$, we have derived the following result.

\begin{lm}\label{lm2.2}
For each interval $[i\delta,(i+1)\delta]$ with $i\in\mathbb N$ and any $f\in L^{\frac 43}([i\delta,(i+1)\delta];L^{\frac 43}),$ it holds that 
\begin{align*}
&\Big\|\pi_N\int_{i\delta}^{t} e^{\bi(B^{\delta}(t)-B^{\delta}(s))\Delta}  f(s)ds\Big\|_{L^{4}([i\delta,(i+1)\delta];L^{4})}\\
&\lesssim (\delta+\frac {\delta}{|B((i+1)\delta)-B(i\delta)|})^{\frac 14}|\log(N)|^{\frac 12}\|f(s)\|_{L^{\frac 43}([i\delta,(i+1)\delta];L^{\frac 43})}, \text{a.s.}
\end{align*}
\end{lm}

Next we are in a position to show global-wellposedness of the Wong--Zakai approximation \eqref{wong-zakai}.

\begin{tm}\label{tm-wong-zakai}
Let $T>0$, $N\in \mathbb N^+$ and $u_0\in L^2.$ 
For almost every fixed $\omega\in \Omega$, there exists $\delta>0$ small enough such that the unique mild solution of \eqref{wong-zakai} exists.
Furthermore, it satisfies the mass conservation law 
$$\|u^{\delta,N}(t)\|_{L^2}=\|u_0\|_{L^2},$$
for any $t\in [0,T].$
\end{tm}

\begin{proof}
For convenience, we denote $M$ such that $M\delta \le T\le (M+1)\delta$
with $\delta=\delta(\omega)>0$ being determined later.
First, we show that the map defined by 
\begin{align*}
\Gamma (\psi)(t):=e^{\bi B^{\delta}(t) \Delta}\pi_N u_0+\bi \pi_N \int_0^t e^{\bi (B^{\delta}(t)-B^{\delta}(s))\Delta}|\psi(s)|^2\psi(s)ds, \; \forall\; t\in [0,\delta],
\end{align*}
has the unique fixed point which is also the local solution on a small interval $[0,\delta]$. We claim that the map $\Gamma$ is well-defined on $L^4([0,\delta],L^4)$. Indeed, by Lemmas \ref{sto-group} and \ref{lm2.2}, 
for $\psi\in L^4([0,\delta],L^4)\cap \mathcal C([0,\delta]; L^2)$ with $\|\psi\|_{L^4([0,\delta],L^4)}\le R$ with some large $R>\max(1,\|u_0\|_{L^2})$, 
\begin{align*}
&\|\Gamma(\psi)\|_{L^4([0,\delta];L^4)}\\
&\le \|e^{\bi B^{\delta}(t)\Delta}u_0\|_{L^4([0,\delta];L^4)}
+|\lambda|\Big\|\int_0^{t}e^{\bi(B^{\delta}(t)-B^{\delta}(s))\Delta} \pi_{N} |\psi(s)|^2\psi(s) ds\Big\|_{L^4([0,\delta];L^4)}
\\
&\le C_1(\delta+\frac {\delta}{|B(\delta)|})^{\frac 14}\|u_0\|_{L^2(\mathbb T)}
+C_2(\delta+\frac {\delta}{|B(\delta)|})^{\frac 14}|\log(N)|^{\frac 12}
\|\psi\|_{L^4([0,\delta];L^4)}^{3}<+\infty.
\end{align*}
By choosing 
\begin{align}\label{cond-1}
    C_1(\delta+\frac {\delta}{|B(\delta)|})^{\frac 14} \le \frac 1 2, C_2(\delta+\frac {\delta}{|B(\delta)|})^{\frac 14}|\log(N)|^{\frac 12}R^2\le \frac 12,
\end{align}
 we have verified that $\Gamma$ maps the ball in $L^4([0,\delta],L^4)$ into itself. 
The contraction property of $\Gamma$ also holds  since 
\begin{align*}
 \|\Gamma(\psi_1)-\Gamma(\psi_2)\|_{L^4([0,\delta];L^4)}
 &\le C_3 (\delta+\frac {\delta}{|B(\delta)|})^{\frac 14}|\log(N)|^{\frac 14}R^2
\|\psi_1-\psi_2\|_{L^4([0,\delta];L^4)},
\end{align*}
by requiring  
\begin{align}\label{cond-2}
   C_3 (\delta+\frac {\delta}{|B(\delta)|})^{\frac 14}|\log(N)|^{\frac 12}R^2<1. 
\end{align}
As a consequence, if \eqref{cond-1} and \eqref{cond-2} hold, then a standard fixed point  argument leads to the local well-posedness of Galerkin approximation. 

We denote the solution on $[0,\delta]$ by $u^{\delta,N}.$ The mass conservation law on $[0,\delta]$ holds due to the chain rule.
To extend the local solution to the global one, we need repeat the above procedures and make use of the fact that $B((i+1)\delta)-B(i\delta)$ shares the same distribution of $B(\delta)$ for any $i\in \mathbb N.$
It suffices to verify the conditions like \eqref{cond-1} and \eqref{cond-2} in each subinterval, i.e.,
\begin{align}\label{sim-cond1}
&C_4 \delta^{\frac 14}|\log(N)|^{\frac 12}R^2\le \frac 12, \\\label{sim-cond2}
&C_4|\log(N)|^{\frac 12}R^2 \frac {\delta^\frac 18}{\log(\frac 1\delta)^{\frac 18}} \le \frac 12 \Big(\frac {|B((i+1)\delta)-B(i\delta)|}{\sqrt{\delta \log(\frac 1\delta)}}\Big)^{\frac 14},
\end{align}
where $C_4=\max(C_1,C_2,C_3).$
Since 
\begin{align*}
 \inf_{i=1}^M \frac {|B((i+1)\delta)-B(i\delta)|}{\sqrt{\delta \log(\frac 1\delta)}}>0, a.s.,   
\end{align*}
 the condition \eqref{sim-cond2} holds for $\delta>0$ sufficiently small. We can let $\delta^{\frac 18} |\log(N)|^{\frac 12} R^2$ small enough such that \eqref{sim-cond1} also holds.
\end{proof}

To end this part, we would like to point out that there is no a priori uniform bound under the norm $L^4([0,T];L^4).$ By using similar arguments in the proof of Theorem \ref{tm-wong-zakai}, one can show that if $u_0\in H^{\bs}$ with $\bs>\frac d2,$ then $u^{\delta,N}\in \mathcal C([0,T];H^{\bs}).$ But there is still no a priori control on the norm $\mathcal C([0,T];H^{\bs})$ unless the smallness condition is imposed.

\subsection{H\"older estimate of $B^{\delta,R}$ (Proof of \eqref{hold-est})}\label{app:proof_of_hold-est}

This result can be found in various forms in \cite{SL2017JDE,MR183023} and related literature but for completeness we provide a proof below. 

\begin{lm}\label{hold-B}
 Let $T>0,$ $p\ge 1$, $\delta\in [0,1), R\ge \max(\sqrt{4p|\ln(\delta)|},pe^{-1})$. It holds that for any $t,s\in [0,T],$
 \begin{align}\label{hold-est_app}
   \|B^{\delta,R}(t)-B^{\delta,R}(s)\|_{L^p(\Omega)}\lesssim \max(|t-s|^{\frac 12},\delta^{\frac 12}). 
 \end{align}
\end{lm}

\begin{proof}
When $\delta=0,R=+\infty,$ \eqref{hold-est_app} is a known result for the standard Brownian motion. In general case ($\delta>0$), by Lemma 2.1, it suffices to bound the continuity estimate of the standard Wong--Zakai approximation, i.e., 
$\|B^{\delta,+\infty}(t)-B^{\delta,+\infty}(s)\|_{L^p(\Omega)}.$
From the definition of $B^{\delta,+\infty}$, it folows that 
\begin{align*}
 &\|B^{\delta,+\infty}(t)-B^{\delta,+\infty}(s)\|_{L^p(\Omega)} \\
 &\le \|B^{\delta,+\infty}(t)-B(\lfloor \frac {t}{\delta} \rfloor \delta)\|_{L^p(\Omega)}
 +\|B(\lfloor \frac {t}{\delta}\rfloor\delta )-B(\lfloor \frac {s}{\delta}\rfloor\delta )\|_{L^p(\Omega)}
 +\|B(\lfloor \frac {s}{\delta} \rfloor \delta)-B^{\delta,+\infty}(s)\|_{L^p(\Omega)}\\
 &\lesssim \frac {t-\lfloor \frac {t}{\delta}\rfloor\delta}{\delta} \delta^{\frac 12}+ |\lfloor \frac {t}{\delta}\rfloor\delta-\lfloor \frac {s}{\delta}\rfloor\delta|^{\frac 12}+\frac {s-\lfloor \frac {s}{\delta} \rfloor\delta}{\delta} \delta^{\frac 12}.
\end{align*}
This verifies \eqref{hold-est_app} with $R=+\infty$. For the case that $R<\infty$ and $\delta>0$, by using Lemma \ref{lm-2.1-wong-zakai} and \eqref{hold-est_app} with $R=+\infty$, we complete the proof.
\end{proof}

\section{Proof of some lemmas}

\subsection{Proof of Lemma \ref{lm-2.1}}\label{app:proof_of_lem-2.1}

\begin{proof}
    In the Fourier coordinates we have
    \begin{align*}
        &\left\|[e^{\bi \Delta (B(t_1)-B(t_2))}-e^{\bi \Delta (B^{\delta,R}(t_1)-B^{\delta,R}(t_2))}]\psi \right\|_{H^{\bs}}^2\\
        &= \sum_{m\in\mathbb{Z}}\left|e^{-\bi m^2(B(t_2)-B(t_1))}-e^{-\bi m^2(B^{\delta,R}(t_2)-B^{\delta,R}(t_1))}\right|^2|\hat{\psi}_m|^2(1+m^{2\bs}).
    \end{align*}
    Here $\widehat \psi_m$ is the Fourier transformation of $\psi.$ 
    Using the estimate $|e^{\bi a}-1|\leq 2 |a|^{\frac \gamma 2},$ for all $\gamma\in [0,2], a\in\mathbb{R},$ we have
    \begin{align*}
        &\left\|[e^{\bi \Delta (B(t_1)-B(t_2))}-e^{\bi \Delta (B^{\delta,R}(t_1)-B^{\delta,R}(t_2))}]\psi \right\|_{H^{\bs}}^2\\
&\lesssim \sum_{m\in\mathbb{Z}}|B(t_2)-B(t_1)-(B^{\delta,R}(t_2)-B^{\delta,R}(t_1))|^\gamma (1+m^{2(\bs+\gamma)})|\hat{\psi}_m|^2\\
        &\lesssim \big(|B(t_1)-B^{\delta,R}(t_1)|^{\frac{\gamma}{2}} +|B(t_2)-B^{\delta,R}(t_2)|^{\frac{\gamma}{2}}\big)^2 \|\psi\|_{H^{\bs+\gamma}}^2.
 \end{align*}
 This implies the desired result. 
 \end{proof}

\subsection{Proof of Lemma \ref{lm-improv-order}}\label{app:proof_lm-improv-order}

\begin{proof}
For convenience, we only show the detailed proof for the case $N=\infty$. 
Indeed, if $N<\infty$, we can use the following decomposition,
$$ \|u(t)-u^{\delta,R,N}(t)\|_{L^2(\Omega;H^{\bs})}\le\|u(t)- u^{0,+\infty,N}(t)\|_{L^2(\Omega;H^{\bs})}+\|u^{0,+\infty,N}(t)-u^{\delta,R,N}(t)\|_{L^2(\Omega;H^{\bs})} $$
Since the a priori estimates of $u^{\delta,R,N}(t)$ in Theorem 2.1   hold under the smallness condition, one can use the standard arguments for the spectral Galerkin approximation to get
\begin{align*}
     \|u^{0,+\infty,N}(t)-u(t)\|_{L^2(\Omega; H^{\bs})}\lesssim N^{-\gamma},
\end{align*}
if $u_0\in H^{\bs+\gamma}$ for any $\gamma>0.$ To estimate $
  \|u^{0,+\infty,N}(t)-u^{\delta,R,N}(t)\|_{H^{\bs}},
$
the procedures are similar to the case $N=\infty$ below.

For simplicity, we assume that $\frac {T}{\delta}=M$ for some integer $M\in\mathbb N^+.$
Denote $t_n=n\delta$ for $n\le M.$  
We let $N=\infty$ and
 write 
\begin{align*}E_n:=\left\|u(t_n)-u^{\delta,R,+\infty}(t_n)\right\|_{H^{\bs}}.
\end{align*}
Denote the solution flows of the exact solution and the Wong--Zakai approximation by $\widetilde \varphi_{s,t}$ and $\varphi_{s,t}$, respectively. 

Then, we have that for $n\le M-1,$
 \begin{align*}
 \mathbb E [E_{n+1}^2]     &=\mathbb E \left\|\widetilde \varphi_{t_{n},t_{n+1}}(u(t_{n}))-\varphi_{t_{n},t_{n+1}}(u^{\delta,R,+\infty}(t_{n}))\right\|_{H^{\bs}}^2\\
        &\leq \mathbb E \left\|\widetilde \varphi_{t_{n},t_{n+1}}(u(t_{n}))-\varphi_{t_{n},t_{n+1}}(u(t_{n}))\right\|_{H^{\bs}}^2\\
&\quad+\mathbb E \left\|\varphi_{t_{n},t_{n+1}}(u(t_{n}))-\varphi_{t_{n},t_{n+1}}(u^{\delta,R,+\infty}(t_{n}))\right\|_{H^{\bs}}^2\\
&+2\mathbb E \<\widetilde \varphi_{t_{n},t_{n+1}}(u(t_{n}))-\varphi_{t_{n},t_{n+1}}(u(t_{n})), \\
&\quad \quad \varphi_{t_{n},t_{n+1}}(u(t_{n}))-\varphi_{t_{n},t_{n+1}}(u^{\delta,R,+\infty}(t_{n}))\>_{H^{\bs}}\\
&=:\mathbb E[\|\mathcal K_1\|_{H^{\bs}}^2]+\mathbb E[\|\mathcal K_2\|_{H^{\bs}}^2]+2\mathbb E[\<\mathcal K_1,\mathcal K_2\>_{H^{\bs}}].
\end{align*}
On the one hand, similarly to the proof of Proposition \ref{prop:pathwise_local_error}, using the a priori estimate in Theorem \ref{small-wel}, we have that for some constant $C>0,$ 
\begin{align}\nonumber
&\E [\|\mathcal K_1\|_{H^{\bs}}^2] \\\label{est-k1}
&\le 2\E \Big[\Big\|\Big(e^{\bi (B(t_{n+1})-B(t_{n})))\Delta }-e^{\bi (B^{\delta,R}(t_{n+1})-B^{\delta,R}(t_{n})))\Delta } \Big)u(t_{n})\Big\|_{H^{\bs}}^2\Big]\\\nonumber
&+2\E \Big[\Big\|\bi \lambda \int_{t_{n}}^{t_{n+1}} \Big(e^{\bi (B(s)-B(t_{n})))\Delta}|e^{\bi (B(s)-B(t_{n})))\Delta}u(t_{n})|^2e^{\bi (B(s)-B(t_{n})))\Delta}u(t_{n})\\\nonumber
&\qquad -
e^{\bi (B^{\delta,R}(s)-B^{\delta,R}(t_{n})))\Delta}|e^{\bi (B^{\delta,R}(s)-B^{\delta,R}(t_{n})))\Delta}u(t_{n})|^2e^{\bi (B^{\delta,R}(s)-B^{\delta,R}(t_{n})))\Delta}u(t_{n})ds\Big\|_{H^{\bs}}^2 \Big]\\\nonumber
&+C\delta^4=: 2\E[\|\mathcal K_{1,1}\|_{H^{\bs}}^2]+2\E [\|\mathcal K_{1,2}\|^2_{H^{\bs}}]+C\delta^4.
\end{align}
According to the definition of $B^{R,\delta}$ \eqref{eqn:linear_WZ_approx} and \eqref{err-trun} in the proof of Lemma \ref{lm-2.1-wong-zakai}, by a interpolation inequality, it holds that 
\begin{align*}
    \E [\|\mathcal K_{1,1}\|_{H^{\bs}}^2]\lesssim e^{-\frac {R^2}{2}\gamma}\lesssim  \delta^{2\gamma}
\end{align*}
since $R\ge \sqrt{8|\ln(\delta)|}.$

To deal with $\E [\|\mathcal K_{1,2}\|^2_{H^{\bs}}]$, we consider the error in Fourier modes. By Parserval's equality and H\"older's equality, it holds that 
\begin{align*}
&\E [\|\mathcal K_{1,2}\|^2_{H^{\bs}}]\\
&\le |\lambda|^2\delta 
\sum_{k\in \mathbb 
 Z}\int_{0}^{\delta}  \E\Big[ \Big| \sum_{k+k_1=k_2+k_3}\Big(e^{-\bi(B (t_{n}+s)-B(t_{n}))(k^2+k_1^2-k_2^2-k_3^2)} \\
 &\quad\quad\quad\quad\quad\quad\quad\quad\quad\quad\quad\quad\quad\quad\quad\quad-e^{-\bi(B^{\delta,R} (t_{n}+s)-B^{\delta,R}(t_{n}))(k^2+k_1^2-k_2^2-k_3^2)}\Big)\\
 &\quad \times \overline{(\widehat{u(t_{n})})_{k_1}}(\widehat{u(t_{n})})_{k_2}(\widehat{u(t_{n})})_{k_3}\Big|^2 \Big]ds (1+|k|^{2\bs}) .
\end{align*}

By using the properties of the   conditional expectation, 
\begin{align*}
\E [\|\mathcal K_{1,2}\|^2_{H^{\bs}}]
&\le |\lambda|^2\delta 
\sum_{k\in \mathbb 
 Z} (1+|k|^{2\bs}) \E\Big[\int_0^{\delta} \sum_{k+k_1=k_2+k_3,k+l_1=l_2+l_3}  \E [\mathcal K_{1,2}^{k,k_i,l_i}|\mathcal F_{t_{n}}] ds\\
 &\quad \times |(\widehat{u(t_{n})})_{k_1}||(\widehat{u(t_{n})})_{k_2}| |(\widehat{u(t_{n})})_{k_3}|  |(\widehat{u(t_{n})})_{l_1}||(\widehat{u(t_{n})})_{l_2}| |(\widehat{u(t_{n})})_{l_3}| \Big].
\end{align*}
Here we denote \begin{align*}
 \mathcal K_{1,2}^{k,k_i,l_i}&:=\Big(e^{\bi(B (t_n+s)-B(t_n))(k^2+k_1^2-k_2^2-k_3^2)} -e^{\bi(B^{\delta,R} (t_n+s)-B^{\delta,R}(t_n))(k^2+k_1^2-k_2^2-k_3^2)}\Big)\\
 &\times \Big(e^{-\bi(B (t_n+s)-B(t_n))(k^2+l_1^2-l_2^2-l_3^2)} -e^{-\bi(B^{\delta,R} (t_n+s)-B^{\delta,R}(t_n))(k^2+l_1^2-l_2^2-l_3^2)}\Big)\\
 &= \Big[e^{\bi(B (t_n+s)-B(t_n))(k_1^2-k_2^2-k_3^2-l_1^2+l_2^2+l_3^2)}\\
 &+e^{\bi(B (t_n+s)-B(t_n))(k^2+k_1^2-k_2^2-k_3^2)-\bi(B^{\delta,R} (t_n+s)-B^{\delta,R}(t_n)) (k^2+l_1^2-l_2^2-l_3^2)}\\
 &-e^{\bi(B^{\delta,R} (t_n+s)-B^{\delta,R}(t_n))(k^2+k_1^2-k_2^2-k_3^2)-\bi(B (t_n+s)-B(t_n))(k^2+l_1^2-l_2^2-l_3^2)}\\
 &-e^{\bi(B^{\delta,R} (t_n+s)-B^{\delta,R}(t_n))(k_1^2-k_2^2-k_3^2-l_1^2+l_2^2+l_3^2)}\Big].
 \end{align*}
Applying  independent increment properties of the Brownian motion, \eqref{hold-est_app}, and Lemma \ref{lm-2.1-wong-zakai}, and using the Taylor expansion, 
one has that 
\begin{align*}
    \Big|\E [\mathcal K_{1,2}^{k,k_i,l_i}|\mathcal F_{t_{n-1}}]\Big|&\lesssim (e^{-\frac {R^2}2}+\delta)[(k^2+k_1^2-k_2^2-k_3^2)^2+(k^2+l_1^2-l_2^2-l_3^2)^2]\\
    &\lesssim   \delta [(k^2+k_1^2-k_2^2-k_3^2)^2+(k^2+l_1^2-l_2^2-l_3^2)^2].
\end{align*}
As a consequence, for $\gamma>0,$
\begin{align}\label{est-k12}
  \E [\|\mathcal K_{1,2}\|^2_{H^{\bs}}]\lesssim \delta^{2+\frac {\max(\gamma,4)}4} \|u(t_{n})\|^3_{\mathcal H^{\bs+\gamma}}.
\end{align}
This, together with the estimate of $\E[\|\mathcal K_{1,1}\|_{H^\bs}^2]$ and \eqref{est-k1}, yields that 
\begin{align}\label{final-est-k1}
 \E[\|\mathcal K_{1}\|_{H^\bs}^2] 
 \lesssim \delta^{2+\frac {\max(\gamma,4)}4}. 
\end{align}

On the other hand, by the similar arguments as in the proof of Proposition \ref{prop:pathwise_stability}, we have that there exists $C_1>0$ such that
\begin{align}\label{final-est-k2}
\E[\|\mathcal K_{2}\|^2_{H^{\bs}}]
&\le (1+C_1\delta) \E [E_{n}^2].
\end{align}

Now 
it suffices to bound $2\mathbb E[\<\mathcal K_1,\mathcal K_2\>_{H^{\bs}(\mathbb T)}]$.
Note that 
\begin{align*}
  \mathcal K_1=\mathcal K_{1,1}+\mathcal K_{1,2}+\mathcal K_{1,3},
\end{align*}
where $\mathcal K_{1,3}$ satisfies 
\begin{align*}
 \mathcal K_{1,3}
 &=\bi \lambda \Big(\int_{t_{n}}^{t_{n+1}} e^{\bi (B(s)-B(t_n)))\Delta}|u(s)|^2u(s)ds\\
 &-\int_{t_{n}}^{t_{n+1}} e^{\bi (B(s)-B(t_n)))\Delta}|e^{\bi (B(s)-B(t_n)))\Delta}u(t_{n})|^2e^{\bi (B(s)-B(t_n)))\Delta}u(t_{n})ds\Big)\\
 &-\bi \lambda \Big(\int_{t_{n}}^{t_{n+1}} e^{\bi (B^{\delta,R}(s)-B^{\delta,R}(t_n)))\Delta}|\phi_{t_{n},s}(u(t_{n}))|^2\phi_{t_{n},s}(u(t_{n}))ds,\\
 &-\int_{t_{n}}^{t_{n+1}} e^{\bi (B^{\delta,R}(s)-B^{\delta,R}(t_n)))\Delta}|e^{\bi (B^{\delta,R}(s)-B^{R,\delta}(t_n)))\Delta}u(t_{n})|^2e^{\bi (B^{\delta,R}(s)-B^{\delta,R}(t_n)))\Delta}u(t_{n})ds\Big).
\end{align*}

Then we can see that 
\begin{align*}
  & 2\mathbb E[\<\mathcal K_1,\mathcal K_2\>_{H^{\bs}}]=2\mathbb E \Big\<\mathcal K_{1,1} +\mathcal K_{1,2}+\mathcal K_{1,3},\mathcal K_{2,1}
  +\mathcal K_{2,2}\Big\>_{H^s_k},
\end{align*}
where  
\begin{align*}
 \mathcal K_{2,1}&=e^{\bi (B^{\delta,R}(t_{n+1})-B^{\delta,R}(t_{n})))\Delta } u(t_{n})-e^{\bi (B^{\delta,R}(t_{n+1})-B^{\delta,R}(t_{n})))\Delta } u^{\delta,R,N}(t_{n})\\
 \mathcal K_{2,2}  &= \bi \lambda \Big(\int_{t_{n}}^{t_{n+1}} e^{\bi (B^{\delta,R}(s)-B^{\delta,R}(t_n)))\Delta}|\phi_{t_{n},s}(u(t_{n}))|^2\phi_{t_{n},s}(u(t_{n}))ds\\
  &- e^{\bi (B^{\delta,R}(s)-B^{\delta,R}(t_n)))\Delta}|\phi_{t_{n},s}(u^{\delta,R,N}(t_{n}))|^2\phi_{t_{n},s}(u^{\delta,R,N}(t_{n}))ds\Big).
\end{align*}

We decompose this as 
\begin{align}\label{err-decomp-cross}
  2\mathbb E[\<\mathcal K_1,\mathcal K_2\>_{H^{\bs}}]&=2 \mathbb E \<\mathcal K_{1,2},\mathcal K_{2,1}  \>_{H^{\bs}}\\\nonumber
  &+2 \mathbb E \<\mathcal K_{1,2}, \mathcal K_{2,2} \>_{H^\bs}\\\nonumber
  &+2 \mathbb E \<K_{1,1}+\mathcal K_{1,3},\mathcal K_{2,1}\>_{H^\bs} \\\nonumber
  &+ 2 \mathbb E \<K_{1,1}+\mathcal K_{1,3},\mathcal K_{2,2}\>_{H^\bs}.
\end{align}
By the a priori bounds of $\phi_{t_{n},s}(u(t_{n}))$ and $\tilde \phi_{t_{n},s}(u(t_{n}))$, as well the Young's inequality,
we have that 
\begin{align*}
&2 \mathbb E \<K_{1,1}+\mathcal K_{1,3},\mathcal K_{2,1}\>_{H^\bs}\lesssim \delta \E [E_{n}^2]+\delta^3,
\end{align*}
and 
\begin{align*}
2 \mathbb E \<K_{1,1}+\mathcal K_{1,3},\mathcal K_{2,2}\>_{H^\bs}&\lesssim \delta^3.
\end{align*}
Similarly, by using \eqref{est-k12}, we can obtain that
\begin{align*}
  &2 \mathbb E \<\mathcal K_{1,2}, \mathcal K_{2,2} \>_{H^\bs} \lesssim \delta \mathbb E[E_{n}^2]+\delta^3.
\end{align*}
Finally, we estimate the term $2 \mathbb E \<\mathcal K_{1,2},\mathcal K_{2,1}  \>_{H_k^{\bs}}$. Using the conditional expectation and the independent increments  of Brownian motion, we have that 
\begin{align*}
  &2 \mathbb E \<\mathcal K_{1,2},\mathcal K_{2,1}  \>_{H^{\bs}}\\  &= \sum_{k\in\mathbb Z}\sum_{k+k_1=k_2+k_3}2\E \Re \Big[\E \Big[ \int_{0}^{\delta}\Big(e^{-\bi(B (t_{n}+s)-B(t_{n}))(k^2+k_1^2-k_2^2-k_3^2)}\\
  &\quad\quad-e^{-\bi(B^{\delta,R} (t_{n}+s)-B^{\delta,R}(t_{n}))(k^2+k_1^2-k_2^2-k_3^2)}\Big)e^{\bi (B^{\delta,R}(t_{n+1})-B^{\delta,R}(t_{n}))k^2}\Big|\mathcal F_{t_{n}} \Big]\\&\quad\quad\quad\quad\quad\quad\overline{(\widehat{u(t_{n})})_{k_1}}(\widehat{u(t_{n})})_{k_2}(\widehat{u(t_{n})})_{k_3}\overline{(\widehat{u(t_{n})}-\widehat{u^{\delta,R,+\infty}(t_{n}))}_{k}}\Big] ds (1+|k|^{2\bs})\\\nonumber
  & \lesssim \delta \E[E_{n}^2]+\delta^{1+\frac {\max(\gamma,4)}2}.
\end{align*} 
This, together with the estimates of $
  2 \mathbb E \<\mathcal K_{1,2}, \mathcal K_{2,2} \>_{H_k^\bs}
  $, $2 \mathbb E \<K_{1,1}+\mathcal K_{1,3},\mathcal K_{2,1}\>_{H_k^\bs}$, $2 \mathbb E \<K_{1,1}+\mathcal K_{1,3},\mathcal K_{2,2}\>_{H_k^\bs}$ and \eqref{err-decomp-cross}, implies that 
\begin{align}\label{final-err-cross}
   2 \E [\<\mathcal K_1,\mathcal K_2\>_{H^{\bs}}]\lesssim  \delta \E[E_{n}^2]+\delta^{1+\frac {\max(\gamma,4)}2}.
  \end{align}
Finally, combining \eqref{final-est-k1}, \eqref{final-est-k2}, and \eqref{final-err-cross} together, and then applying the Gronwall inequality, we get 
\begin{align}\label{err-in-grid}
    \sup_{n\le M}\E[E_n^2]\lesssim \delta^{\frac {\max(\gamma,4)} 2}.
\end{align}

For any $t\in [0,T],$ there exists an integer $n\ge 1$ such that $t\in [t_{n-1},t_n]$, then we can repeat the similar  procedures in proving the error estimate of $\E[E_n^2]$ and use \eqref{err-in-grid} to obtain that
\begin{align*}
    \sup_{t\in [0,T]}\E[\|u(t)-u^{\delta,R,+\infty}(t)\|^2_{H^{\bs}}]\lesssim \delta^{\frac {\max(\gamma,4)}2},
\end{align*} 
which completes the proof.

\end{proof}

\end{appendix}

\end{document}